\documentclass{article}
\usepackage{color,amsmath,amssymb,amsthm,bm,enumerate,geometry}
\usepackage[dvips]{graphicx}
\usepackage[titletoc,title]{appendix}
\newcommand{\keywords}[1]{\textbf{Key words:} #1}
\newcommand{\msc}[1]{\textbf{MSC2010:} #1}

\newcommand{\ack}{\textbf{Acknowledgment}}

\DeclareMathOperator{\Int}{Int}

\DeclareMathOperator{\Var}{Var}

\DeclareMathOperator{\Supp}{Supp}

\DeclareMathOperator{\TV}{TV}
\DeclareMathOperator{\Diam}{diam}
\DeclareMathOperator{\MT}{T_{mix}}
\def\ER{R_{\rm{eff}}} 
\def\ERN{R_{\rm{eff}}^{(N)}} 
\def\TERN{\widehat{R}_{\rm{eff}}^{(N)}}

\def\CT{T_{\rm{cov}}} 
 
\def\SCT{\tau_{\rm{cov}}}

\def\MT{T_{\rm{mix}}}

\def\CPHI{C_{\rm{PHI}}}
\def\CHK{C_{\rm{HK}}}  
\def\CCS{C_{\rm{CS}}}
\def\CPI{C_{\rm{PI}}}
\def\CQI{C_{\rm{QI}}}
\def\cHK{c_{\rm{HK}}}
\def\cPHI{c_{\rm{PHI}}}
\def\creff{c_{\rm{R}}}
\def\cFK{c_{\rm{FK}}}
\def\cgr{c_{\rm{g}}}
\def\CD{C_{\rm{D}}}
\def\CV{C_{\rm{V}}}
\def\ce{c_{\rm{e}}}
\def\cv{c_{\rm{v}}}
\def\R{\mathbb{R}}

\def\rnum#1{\expandafter{\romannumeral #1}} 
\def\Rnum#1{\uppercase\expandafter{\romannumeral #1}} 
\newcommand{\abbr}[1]{{\small\sc\lowercase{#1}}}

\definecolor{red}{rgb}{1,0,0}
\definecolor{green}{rgb}{0,1,0}
\definecolor{blue}{rgb}{0,0,1}
\definecolor{cyan}{cmyk}{1,0,0,0}
\definecolor{magenta}{cmyk}{0,1,0,0}
\definecolor{yellow}{cmyk}{0,0,1,0}
\definecolor{purple}{rgb}{0.9,0,0.8}
\definecolor{gray}{rgb}{0.7,0.7,0.7}

\title{Cutoff for lamplighter chains on fractals}
\date{July 15th, 2018}

\author{ 
 Amir Dembo
 \thanks{Department of Mathematics and Department of Statistics, Stanford University
   Sequoia Hall, 390 Serra Mall Stanford, California 94305-4065 USA ~~~~~ E-mail:{\tt  amir@math.stanford.edu}}
  \and 
 Takashi Kumagai
  \thanks{Research Institute for Mathematical Sciences,
Kyoto University, Kyoto 606-8502, Japan. \newline ~~~~~E-mail: {\tt kumagai@kurims.kyoto-u.ac.jp}}
\and Chikara Nakamura
\thanks{Research Institute for Mathematical Sciences,
Kyoto University, Kyoto 606-8502, Japan. \newline ~~~~~E-mail: {\tt chikaran@kurims.kyoto-u.ac.jp}}
}

\begin{document}

\maketitle
\footnote[0]{ 
   \\
 \msc{Primary: 60J10; Secondary: 28A80, 35K08.}  \\
}
\makeatother

\newtheorem{Definition}{Definition}[section]
\newtheorem{Proposition}[Definition]{Proposition}
\newtheorem{Theorem}[Definition]{Theorem}
\newtheorem{Assumption}[Definition]{Assumption}
\newtheorem{Lemma}[Definition]{Lemma}
\newtheorem{Remark}[Definition]{Remark}
\newtheorem{Example}[Definition]{Example}
\newtheorem{Corollary}[Definition]{Corollary}
\newtheorem{Notation}[Definition]{Notation}
\newtheorem{Application}[Definition]{Application}
\newtheorem{Condition}[Definition]{Condition}

\makeatletter
\@addtoreset{equation}{section}
\def\theequation{\thesection.\arabic{equation}}

\begin{abstract}
We show that the total-variation mixing time of the
lamplighter random walk on fractal graphs exhibit
sharp cutoff when the underlying graph is transient
(namely of spectral dimension greater than two). 
In contrast, we show that such cutoff can not occur 
for strongly recurrent underlying graphs (i.e. of 
spectral dimension less than two).   
\end{abstract}

\section{Introduction}
Markov chain mixing rate is an active subject of study in probability theory 
(see \cite{LPW08,SC97} and the references therein). Mixing 
is usually measured in terms of total variation distance, which for
probability measures $\mu, \nu$ on a countable set $H$ is  
     \begin{align*}
          \| \mu - \nu \|_{\rm{TV}} := \sup_{A \subset H} | \mu (A) - \nu (A) | 
                                =  \frac{1}{2}  \sum_{x \in H} | \mu (x) - \nu (x)| 
                               =  \sum_{x \in H}  [ \mu (x) - \nu (x)]_{+}\,. 
     \end{align*}
Specifically, the ($\epsilon$-)total variation mixing time 
of a Markov chain $Y = \{ Y_t \}_{t \ge 0}$ on the set of vertices 
of a finite graph $G=(V,E)$, having the 
invariant distribution $\pi$, is  
      \begin{align}
                 &\MT (\epsilon ;G) := \min \Big\{ t \ge 0 \Bigm| \max_{x \in V(G)} \| P_x ( Y_t = \cdot ) - \pi \|_{\rm{TV}} \le \epsilon \Big\} .  \notag  
      \end{align} 
One of the interesting topics in the study of Markov chains is the  
cutoff phenomena, mainly for the total variation mixing time 
(see e.g. \cite[Chapter 18]{LPW08}).
 The study of cutoff phenomena for Markov chains was initiated by Aldous, Diaconis and their collaborators early in 80s, and there has been extensive work in the past several decades. 
Specifically,  
a sequence of Markov chains $\{Y^{(N)}\}_{N \ge 1}$ on the vertices of finite graphs $\{G^{(N)}\}_{N \ge 1}$ has cutoff with threshold $\{ a_N \}_{N \ge 1}$ iff 
       \begin{align*}
                \lim_{N \to \infty} a_N^{-1} \MT (\epsilon ; G^{(N)})  = 1,   \qquad \text{  $ \forall \epsilon \in (0,1)$.}  
       \end{align*}
In the (switch-walk-switch) lamplighter Markov chains, 
each vertex of a locally connected, countable 
(or finite) graph $G= (V,E)$ is equipped with a lamp 
(from $\mathbb{Z}_2 = \{ 0,1 \}$), and a move consists of 
three steps:  

\smallskip
\noindent
(a). The walker turns on/off the lamp at the vertex where he/she is, uniformly at random. 

\smallskip
\noindent
(b). The walker either stays at the same vertex, or moves to a randomly chosen nearest neighbor vertex. 

\smallskip
\noindent
(c). The walker turns on/off the lamp at the vertex where he/she is, 
uniformly at random.   

\medskip
Such a lamplighter chain on the graph $G$ is precisely the random walk on 
the corresponding wreath product $G^{\ast} = \mathbb{Z}_2 \wr G$ 
(see Section  \ref{Subsec:Prel} for the precise definitions), 
and the total variation mixing time of a lamplighter chain is 
closely related to the expected cover time of the underlying 
graph $G$, denoted hereafter by $T_{\rm{cov}} (G)$. 
The study of cutoff for lamplighter chains goes back to 
H\"{a}ggstr\"{o}m and Jonasson  \cite{HJ97} who showed 
that cutoff does not occur for the chain on one-dimensional tori, 
whereas for lamplighter chains on complete graphs, 
it occurs at the threshold $a_N = \frac{1}{2} T_{\rm{cov}} (G^{(N)})$.
Peres and Revelle \cite{PR04} further explore the relation between 
the mixing time of lamplighter chain on $G^{(N)}$ and 
$T_{\rm{cov}} (G^{(N)})$, showing that, under suitable assumptions, 
         \begin{align}  \label{eq:LampMix}
                \big(  \frac{1}{2}  + o(1)  \big) \CT (G^{(N)}) \le T_{\rm{mix}} (\mathbb{Z}_2 \wr G^{(N)}; \epsilon)   \le (1 + o(1) ) \CT (G^{(N)}). 
         \end{align}
The bounds of \eqref{eq:LampMix} cannot be improved in general, as the 
lower and the upper bounds are achieved for complete graphs, and 
two-dimensional tori, respectively. The same bounds apply 
for any Markov chain on $\mathfrak{X} \wr G^{(N)}$, where in steps (a) 
and (c) the walker independently chooses the element from the 
finite set $\mathfrak{X}$ according to some fixed strictly positive law. 
Indeed, for such chains total variation mixing time has mostly to do 
with the geometry of late points of $G$, namely those reached by the 
walker much later than most points. In particular, the \abbr{lhs} of \eqref{eq:LampMix}
represents the need to visit all but $O(\sqrt{\sharp V(G)})$ 
points before mixing of the lamps can occur and the \abbr{rhs} 
reflects having the lamps at the invariant product measure once 
all vertices have been visited. Miller and Peres \cite{MP12} 
provide a large class of graphs for which the \abbr{lhs} of 
\eqref{eq:LampMix} is sharp, with 
cutoff at $\frac{1}{2} \CT (G^{(N)})$. Among those 
are lazy simple random walkers on $d$-dimensional tori, any $d \ge 3$,
for which \cite{MS17} further examines the total-variation 
distance between the law of late points and  
i.i.d. Bernoulli points (c.f. 
\cite[Section 1]{MS17} and the references therein). Finally,
the analysis of effective resistance on $G^{(N)} = 
\mathbb{Z}_N^2 \times \mathbb{Z}_{[h \log N]}$ plays a key role in 
\cite{DDMP16}, where it is shown that the threshold 
$a(h) \CT (G^{(N)})$ for mixing time cutoff of 
lamplighter chain on such graphs, continuously interpolates 
between $a(0)=1$ and $a(\infty)=\frac{1}{2}$.

\medskip  
Another topic of much current interest is the long time asymptotic 
behavior of random walks $\{X_t\}$ on (infinite) fractal graphs
(see \cite{Barlow98,Kigami01,Kumagai14} and the references therein).    
Such random walks are typically anomalous and sub-diffusive, so
generically $E_x [ d(X_0, X_t) ] \asymp t^{1/d_w}$ and the 
\emph{walk-dimension $d_w$} exceeds two for many fractal graphs,
in contrast to the \abbr{srw} on $\mathbb{Z}^d$ for which
$d_w=2$ (the notation $a_t \asymp b_t$ is used hereafter 
whenever $c^{-1} a_t \le b_t \le c a_t$ for some $c<\infty$). 
A related important parameter is the \emph{volume growth exponent $d_f$} 
such that $\sharp B(x,r)\asymp r^{d_f}$, where $\sharp B(x,r)$ 
counts the number of vertices whose graph distance from $x$ is 
at most $r$. The growth of the eigenvalues of the corresponding 
generator is then measured by the \emph{spectral dimension $d_s := 2d_f/d_w$,}
with the Markov chain $\{X_t\}$ strongly recurrent when 
$d_s<2$ and transient when $d_s>2$ 
(while $d_f=d_s=d$ for the \abbr{srw} on $\mathbb{Z}^d$).

\medskip
We study here the cutoff for total variation mixing time of the 
lamplighter chain when $G^{(N)}$ are increasing finite subsets 
of a fractal graph. While gaining important insights on the 
geometry of late points for the corresponding walks, our main result 
(see Theorem \ref{Thm:Main}), is the following dichotomy:
\begin{itemize}
            \item  When $d_s < 2$ there is no cutoff for 
             the corresponding lamplighter chain, whereas 
     \end{itemize}
     \begin{itemize}
             \item  if $d_s > 2$, such cutoff occurs at the
             threshold $a_N = \frac{1}{2} T_{\rm{cov}} (G^{(N)})$.  
     \end{itemize}
If the behavior of the lamplighter chain in the critical case 
$d_s=2$ 
is likewise universal, then it should be having a mixing cutoff at  
$a_N=T_{\rm{cov}} (G^{(N)})$ (as in the two-dimensional tori 
example from \cite{PR04}).

\subsection{Framework and main results}  \label{Subsec:Prel}
Given a countable, locally finite and  connected graph $G = (V(G), E(G))$, 
denote by $d (\cdot, \cdot) =d_G (\cdot, \cdot)$ the graph distance 
(with $d(x,y)$ the length of the shortest path between $x$ and $y$), 
and by $B(x,r) = B_G (x,r) := \{ y \in V(G) \mid d(x,y) \le r \}$ the 
corresponding ball of radius $r$ centered at $x$. A weighted graph 
is a pair $(G,\mu)$ with $\mu : V(G) \times V(G) \to [0, \infty )$ 
a conductance, namely a function $(x,y) \mapsto \mu_{xy}$ such that 
$\mu_{xy} = \mu_{yx}$ and $\mu_{xy} > 0$ if and only if $xy \in E(G)$. 
We use the notation $V(x,r) := \mu (B(x,r))$ and more generally
$\mu (A) := \sum_{x \in A} \mu_x$ for $A \subset V(G)$, where  
     \begin{align}  \label{eq:Weight}
\mu_x := \sum_{y :xy \in E(G)} \mu_{xy},  \qquad \text{ $ \forall  x \in V$}.  
     \end{align}

The discrete time random walk $X = \{ X_t \}_{t \ge 0}$ associated with the weighted graph $(G,\mu)$ is the Markov chain on $V(G)$ having the
transition probability 
       \begin{align*}
                P(x,y) :=   \frac{\mu_{xy}}{ \mu_x } \,.  
       \end{align*} 
Let $P_t(x,y) = P_t (x,y; G) := P_x (X_t = y)$ denote the 
distribution of $X_t$ with the corresponding heat kernel 
          \begin{align*}
                       p_t (x,y) := \frac{ P_t (x,y) }{ \mu_y }   \qquad 
                       \forall t \in \mathbb N \cup \{0\}
          \end{align*}
and Dirichlet form             
         \begin{align*}
                 \mathcal{E} (f,f)   & := \frac{1}{2}  \sum_{x,y \in V(G)} (f(x) - f(y) )^2 \mu_{xy}   = - \langle f, (P-I)f \rangle_{\mu},   
                 \quad   \text{for $f : V(G) \to \mathbb{R}$} ,
         \end{align*}
where $\langle f, g\rangle_{\mu} := \sum_{x} f(x) g(x) \mu (x)$.
The corresponding effective resistance 
$\ER (\cdot , \cdot )$ is given by  
         \begin{align*}
                \ER (A,B)^{-1} := \inf \{   \mathcal{E} (f,f)  \mid  f|_A = 1, f|_B = 0 \} , \quad   \text{for $A, B \subset V(G)$} .
          \end{align*}           
We also consider the lazy random walk 
$\tilde{X} = \{ \tilde{X}_t \}_{t \ge 0}$ on 
$(G,\mu)$, having the transition probability 
         \begin{align*}  
                  \tilde{P} (x,y) :=   \begin{cases}    \frac{1}{2} P(x,y),    &  \text{ if $x \neq y$,}  \\   \frac{1}{2},   &  \text{ if $x = y$}.  \end{cases}
         \end{align*} 
The Dirichlet form and heat kernel of $\tilde{X}$ are then, respectively
 $\tilde{\mathcal{E}}(f,f)=\frac{1}{2}\mathcal{E}(f,f)$ 
and    
\begin{align*}
\tilde{p}_t (x,y) 
:= \frac{ \tilde{P}_t (x,y) }{ \mu_y }   \qquad  \qquad  
\forall t\in \mathbb N\cup \{0\} \,.
\end{align*} 
We consider finite weighted graphs 
$\{ (G^{(N)}, \mu^{(N)} ) \}_{N \ge 1}$ with 
$\sharp V(G^{(N)}) \to \infty$. Using hereafter 
$\cdot^{(N)}$ for objects on $(G^{(N)}, \mu^{(N)})$ (e.g.
denoting by $\ER^{(N)}(\cdot,\cdot)$ the effective 
resistance on $(G^{(N)},\mu^{(N)})$), we make the following 
assumptions, which are standard 
in the study of sub-Gaussian heat kernel estimates (\abbr{sub-GHKE}) 
(c.f. \cite{Barlow17,Kumagai14}).   
\begin{Assumption} \label{Ass:Weight}
For some $1 \le d_f  < \infty$, $\ce, \cv  < \infty$, $p_0 > 0$ 
and all $N \ge 1$ we have
        \begin{enumerate}   \renewcommand{\labelenumi}{(\alph{enumi})}  
               \item Uniform ellipticity: \;\; $\ce^{-1} \le \mu_{xy}^{(N)} \le \ce$ \quad  $\forall xy \in E(G^{(N)})$.  

               \item  $p_0$-condition: \;\;  $\displaystyle  \frac{\mu^{(N)}_{xy} }{ \mu^{(N)}_x } \ge p_0$ \quad $\forall xy \in E(G^{(N)})$.     

               \item $d_f$-set condition: \;  $\cv^{-1}  r^{d_f} \le V^{(N)} (x,r) \le \cv  r^{d_f}$ \; $\forall x \in V(G^{(N)})$, $1 \le r \le 
               \Diam \{ G^{(N)} \} \to \infty$.   
        \end{enumerate}
\end{Assumption}

\begin{Assumption}[Uniform Parabolic Harnack Inequality]\label{Ass:UPHI}
For some $2 \le d_w < \infty$, 
$\CPHI <\infty$, $\cPHI \in (0,1]$ and all $N \ge 1$, 
whenever $u : [0, \infty) \times V(G^{(N)}) \to [0,\infty )$ satisfies 
              \begin{align}\label{eq:he-GN}
                      u (t+1, x)  - u (t,x)   = (P^{(N)} - I) u(t,x),  \qquad 
                      \forall t \in [0,4T], x \in B^{(N)} (x_0,2R) \,,  
              \end{align}
for some $x_0 \in V(G^{(N)})$, $R \le \cPHI  \Diam \{ G^{(N)} \}$ 
and $T \ge 2R$, $T \asymp R^{d_w}$, one also has that 
              \begin{align*}  
\max_{ \substack{ z \in B^{(N)}(x_0,R) \\ s \in [T, 2T]} }  \{u(s,z)\}  
                             \le  \CPHI   \min_{ \substack{ z \in B^{(N)} (x_0, R) \\ s \in [3T, 4T]} } \{ u(s,z) + u(s+1,z) \} .  
              \end{align*}  
\end{Assumption}

\begin{Remark}   \label{Rem:Card}
Thanks to the $p_0$-condition we have that $1 \ge \deg_{G^{(N)}}(x) p_0$, so 
the graphs $\{G^{(N)}\}$ are of uniformly bounded degrees 
          \begin{align*}
                     \sup_N   \sup_{x \in V(G^{(N)})} \{ \deg_{G^{(N)}} (x) \} < \infty.  
           \end{align*}
Together with the uniform ellipticity, this implies that 
for some $\tilde{c} < \infty$  
           \begin{align*}
                     \tilde{c}^{-1}   \le \mu_x^{(N)}  \le  \tilde{c}, 
                 \qquad  \qquad \forall N \ge 1, \;\; x \in V(G^{(N)}), 
           \end{align*}
           and thereby 
           \begin{align}   \label{eq:Card}
                       \tilde{c}^{-1} \, \sharp A  \le  \mu^{(N)} (A)  \le \tilde{c} \, \sharp A,  \qquad \forall N \ge 1, \;\;  A \subset V(G^{(N)}) \,. 
           \end{align}  
\end{Remark}

\medskip
To any finite underlying graph $G=(V,E)$ corresponds the 
wreath product $G^{\ast} = \mathbb{Z}_2 \wr G$ such that 
     \begin{align*}
              & V(G^{\ast}) =  \mathbb{Z}_2^V \times V ,  \\      
              & E(G^{\ast}) = \left\{ \{ (f,x), (g,y) \} \mid  \text{ $f=g$ and 
              $xy \in E$,  or $x=y$ and $f (v) =g (v)$ for $v \neq x$}  \right\}   
     \end{align*}
and we adopt throughout the convention of using $\bm{y} = (f,y)$ 
for the vertices of $\mathbb{Z}_2 \wr G$. The lazy random walk $\tilde{X}$ on 
$(G, \mu)$ induces the switch-walk-switch lamplighter chain, namely the
random walk $Y=\{ Y_t = (f_t, \tilde{X}_t) \}_{t \ge 0}$ on 
$\mathbb{Z}_2 \wr G$ whose transition probability is  
     \begin{align*}
             P^{\ast} ((f,x), (g,y)) = 
                   \begin{cases}
  \frac{1}{2} \tilde{P} (x,y) = \frac{1}{4},
     &  \text{ if $x=y$ and $f (v) = g(v)$ for any $v \neq x$,}  \\ 
                         \frac{1}{4}   \tilde{P} (x,y) = \frac{1}{8} P (x,y),            &  \text{ if $x \neq y$ and $f (v) = g(v)$ for any $v \neq x,y$,}  \\
                         0,                                                                     &  \text{ otherwise.}
                   \end{cases}
     \end{align*}
One way to describe the moves of the Markov chain $Y$ is as done 
before: first $Y$ switches the lamp of the current position,
then moves on $G$ according to $\tilde{P}$, and finally switches 
the lamp on vertex on which it landed. We denote by 
$Y^{(N)} = \{ Y^{(N)}_t = (f_t, \tilde{X}_t^{(N)} ) \}_{t \ge 0}$ 
the lamplighter chain on weighted graphs $(G^{(N)},\mu^{(N)})$, using
$P^{\ast} (\cdot, \cdot ; G)$ whenever we wish to emphasize 
its underlying graph. The invariant (reversible) distribution 
of each $X^{(N)}$, and its lazy version $\tilde X^{(N)}$, is clearly   
      \begin{align}  \label{eq:InvUnder}
               \pi^{(N)} (x )  = \frac{\mu_x^{(N)}}{ \mu^{(N)} (G^{(N)})}\,,  \qquad \forall  x \in V(G^{(N)})  
      \end{align}  
with the corresponding invariant distribution of $Y^{(N)}$ being  
     \begin{align*}  
               \pi^{\ast} ( \bm{y} ; G^{(N)} )  =  2^{-\sharp V(G^{(N)})}  \, \pi^{(N)} (y),  \qquad \text{ $\forall  \bm{y} = (f,y) \in V( \mathbb{Z}_2 \wr G^{(N)})$.}  
      \end{align*}  
   
\bigskip  

We next state our main result.

\begin{Theorem}  \label{Thm:Main}   
Consider lamplighter chains $Y^{(N)}$ whose underlying 
weighted graphs  
$\{(G^{(N)},\mu^{(N)})\}_{N \ge 1}$ 
satisfy Assumptions \ref{Ass:Weight}, \ref{Ass:UPHI}. 

\smallskip
\noindent
(a)  If $d_f < d_w$, then there is no cutoff for the 
         total variation mixing time of $Y^{(N)}$.  

\smallskip
\noindent
(b) If $d_f > d_w$, then the total variation 
         mixing time for $Y^{(N)}$ 
         admits cutoff at $a_N = \frac{1}{2} T_{\rm{cov}} (G^{(N)})$.
\end{Theorem}

Note that for countable, infinite weighted graph 
$(G,\mu)$, having $d_f < d_w$ (resp. $d_f > d_w$),
corresponds to a strongly recurrent (resp. transient), random walk 
$\tilde X$ in the sense of \cite[Definition 1.2]{BCK05} (see
\cite[Theorem 1.3, Proposition 3.5 and Lemma 3.6]{BCK05}). In Section \ref{Sec:qwExam} we provide 
a host of fractal graphs satisfying Assumptions \ref{Ass:Weight} and \ref{Ass:UPHI}, with
the Sierpinski gaskets and the two-dimensional Sierpinski carpets as
typical examples of Theorem \ref{Thm:Main}(a),  while
high-dimensional Sierpinski carpets with small holes serve as 
typical examples of Theorem \ref{Thm:Main}(b).     

\medskip  
In Section \ref{Sec:Pre}, we adapt to the setting of large 
finite weighted graphs, certain consequences of 
Assumptions \ref{Ass:Weight} and \ref{Ass:UPHI} which  
are standard for infinite graphs. In case $d_f<d_w$,
the relevant time scale for the cover time $\tau_{\rm{cov}} (G^{(N)})$ is 
shown there to be 
\begin{align}\label{eq:Time}
T_N :=  (R_N)^{d_w} \qquad \text{ where } \qquad R_N := \Diam \{ G^{(N)} \} \,.  
\end{align}
Applying in Section \ref{Sec:CT} results from Section \ref{Sec:Pre}
that apply for $d_f<d_w$, we derive the following uniform exponential tail 
decay for $\tau_{\rm{cov}} (G^{(N)})/T_N$, which is of  
independent interest.  
\begin{Proposition}\label{Thm:UCT10}  
If Assumptions \ref{Ass:Weight}, \ref{Ass:UPHI} hold 
with $d_f< d_w$, then for some $c_0$ finite and all $t$, $N$,
\begin{align}\label{eq:tcov-spread}
 \sup_{z \in V(G^{(N)})}  
\{ P_z ( \tau_{\rm{cov}} (G^{(N)}) > t) \} \le c_0 e^{-t/(c_0 T_N)} \,.
\end{align}  
\end{Proposition}
\noindent
Starting with all lamps off, 
namely at $Y_0 = \bm{x} := (\bm{0},x)$, 
on the event $\{ \sup_{0 \le s \le t}  d ( \tilde{X}_0, \tilde{X}_s )  \le  
\frac{1}{4} R_N \}$, 
all lamps outside $B^{(N)}(x,\frac{1}{4} R_N)$ 
are off at time $t$. Hence, then 
$\| P_{t}^{\ast}  ( \bm{x}, \cdot ; G^{(N)} ) - \pi^{\ast} (\cdot ; G^{(N)} ) \|_{\TV}$ is still far from $0$. 
Using this observation, we prove in Section \ref{Sec:Rec} 
the following uniform lower bound on the lamplighter chain 
distance from equilibrium at time $t \asymp T_N$.    
\begin{Proposition}\label{Prop:LBTV10}
If Assumptions \ref{Ass:Weight}, \ref{Ass:UPHI} hold,
then for some finite $c_1$, $N_1$, any $t$ and $N \ge N_1$,
\begin{align}\label{eq:lbd-tv}
\max_{\bm{x} \in V(\mathbb{Z}_2 \wr G^{(N)})} \,
\| P_{t}^{\ast}  ( \bm{x}, \cdot ; G^{(N)} ) - \pi^{\ast} (\cdot ; G^{(N)} ) \|_{\TV}  \ge   c_1^{-1} e^{- c_1 t/T_N} - \tilde{c} \, \cv R_N^{-d_f} \,.  
\end{align} 
\end{Proposition}
\noindent
In Proposition \ref{Prop:LampTV} we bound the \abbr{lhs} of 
\eqref{eq:lbd-tv} by $\max_x P_x ( \tau_{\rm{cov}} (G^{(N)}) > t)$
provided $t/S_N$ is large (for $S_N$ of \eqref{eq:Not05}). 
Since $S_N  \asymp T_N$ when $d_f < d_w$,
contrasting Propositions \ref{Thm:UCT10} and \ref{Prop:LBTV10}
yields Theorem \ref{Thm:Main}(a) (c.f. Remark 
\ref{rem:mix} for information about 
$\MT(\epsilon; \mathbb{Z}_2 \wr G^{(N)})/\CT (G^{(N)})$ and lack of
concentration of $\tau_{\rm{cov}} (G^{(N)})/T_N$).

Propositions \ref{Prop:LBTV10} and 
\ref{Prop:LampTV} 
apply also when $d_w<d_f$, but in that case 
$\tau_{\rm{cov}} (G^{(N)}) \ge \sharp V(G^{(N)}) \gg T_N$, 
and the proof of Theorem \ref{Thm:Main}(b), provided in Section \ref{Sec:Tran}, 
amounts to verifying the sufficient 
conditions of \cite[Theorem 1.5]{MP12} 
for cutoff at $\frac{1}{2} T_{\rm{cov}} (G^{(N)})$. 
Indeed, the required 
uniform Harnack inequality follows from the 
\abbr{UPHI} of Assumption \ref{Ass:UPHI}, which 
as we see in Section \ref{Sec:qwExam} is more 
amenable to analytic manipulations than 
the Harnack inequality. 
(Yet, we note that quite recently 
stability of the (elliptic) Harnack inequalities is proved in \cite{BM18}.)  

\section{Cutoff in fractal graphs}   \label{Sec:qwExam} 
We provide here a few examples for which Theorem \ref{Thm:Main} applies, 
starting with the following. 
\begin{figure}[htbp]
  \begin{tabular}{ccc}
 \begin{minipage}{0.3\hsize}
  \begin{center}
   \includegraphics[width=10mm]{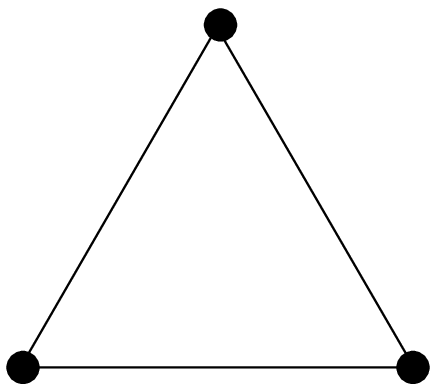}
  \end{center} 
 \end{minipage}
 \begin{minipage}{0.3\hsize}
  \begin{center}
   \includegraphics[width=20mm]{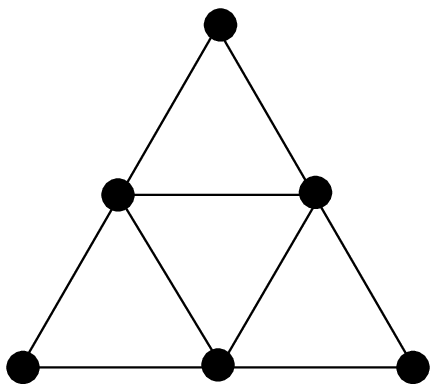}
  \end{center}
 \end{minipage} 
  \begin{minipage}{0.3\hsize}
  \begin{center}
   \includegraphics[width=30mm]{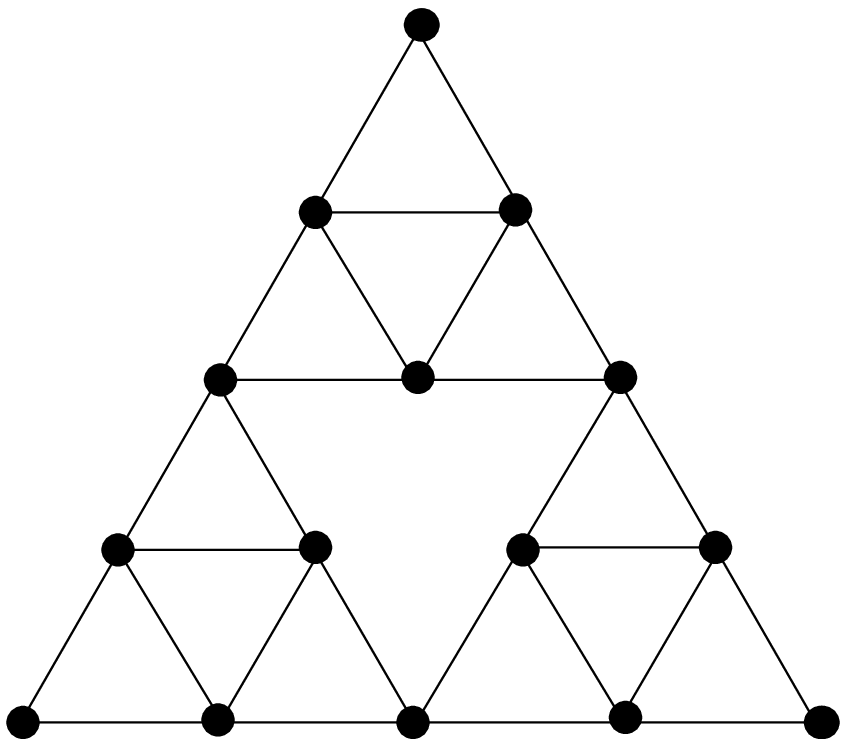}
  \end{center}
 \end{minipage} 
 \end{tabular}  
      \caption{ A sequence of the Sierpinski gasket graphs  ($G^{(0)}, G^{(1)}, G^{(2)}$ respectively).   }
\end{figure}

\begin{Example}[Sierpinski gasket graph in two dimension]\label{Ex:Gasket10}$~$ 
Let $G^{(0)}$ denote the equilateral triangle of side length $1$. That is,  
\begin{align*}
V(G^{(0)}) = \Big\{ x_0=(0,0),   x_1 = (1,0),  x_2 = \big( \frac{1}{2}, \frac{\sqrt{3} }{2} \big) \Big\}, 
\qquad  E(G^{(0)}) = \{ x_0 x_1,  x_0x_2,  x_1x_2 \}. 
\end{align*}      
Setting $\psi_i (x) := (x + x_i)/2$ for $i=0,1,2$, 
we define the graphs $\{G^{(N)}\}_{N \ge 1}$ via 
\[
V(G^{(N+1)}) =2 \cdot \Big(  \bigcup_{i=1}^3  \psi_i (V(G^{(N)}) )   \Big)~~
 \mbox{and}~~
E ( G^{(N+1)} ) =2 \cdot \Big(  \bigcup_{i=1}^3  \psi_i (E(G^{(N)}) )   \Big).
\] 
The limit graph $G = (V(G), E(G))$, where $V(G) = \cup_{N \ge 0} V(G^{(N)})$ 
and $E(G) = \cup_{N \ge 0} E(G^{(N)})$, is called the Sierpinski gasket graph. 
It is easy to confirm that if Assumption \ref{Ass:Weight}(a) holds for 
weight $\mu^{(N)}$ on $G^{(N)}$ then
such $\mu^{(N)}$ satisfies also Assumption \ref{Ass:Weight}(b) 
and Assumption \ref{Ass:Weight}(c) for $d_f = \log 3/\log 2$.
\end{Example}

We further prove in Section \ref{Subsec:pfSC} the following.
\begin{Proposition} \label{thm:SG12}
The weighted graphs $\{ (G^{(N)}, \mu^{(N)}) \}_{N \ge 0}$ 
of Example \ref{Ex:Gasket10} further satisfy Assumption \ref{Ass:UPHI}
with $d_w=\log 5/\log 2$. 
\end{Proposition}
In view of Proposition \ref{thm:SG12} and having $d_f<d_w$, we deduce
from Theorem \ref{Thm:Main}(a) that the total variation mixing time 
of the lamplighter chains of Example \ref{Ex:Gasket10}, admits no cutoff. 
 
\begin{Remark} For $d \ge 3$, the $d$-dimensional Sierpinski gasket graph 
is similarly defined, and by the same reasoning the corresponding 
lamplighter chains admit no 
mixing cutoff. In fact, one can deduce for a more general family 
of nested fractal graphs (see for instance \cite[Section 2]{HK04} 
for definition), that no cutoff applies. 
\end{Remark}

\begin{figure}[htbp]
  \begin{tabular}{ccc}
 \begin{minipage}{0.33\hsize}
  \begin{center}
   \includegraphics[width=20mm]{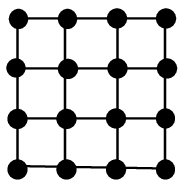}
  \end{center} 
 \end{minipage}
 \begin{minipage}{0.33\hsize}
  \begin{center}
   \includegraphics[width=30mm]{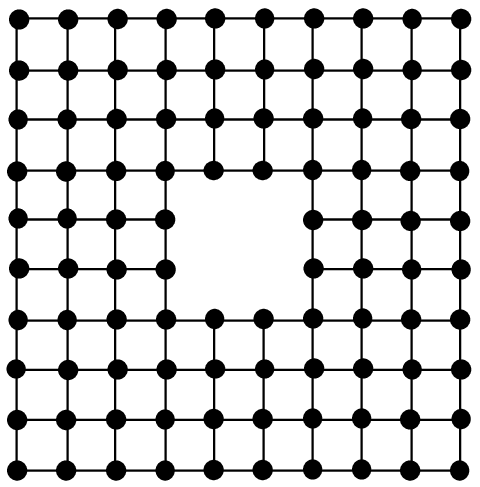}
  \end{center}
 \end{minipage} 
  \begin{minipage}{0.33\hsize}
  \begin{center}
   \includegraphics[width=40mm]{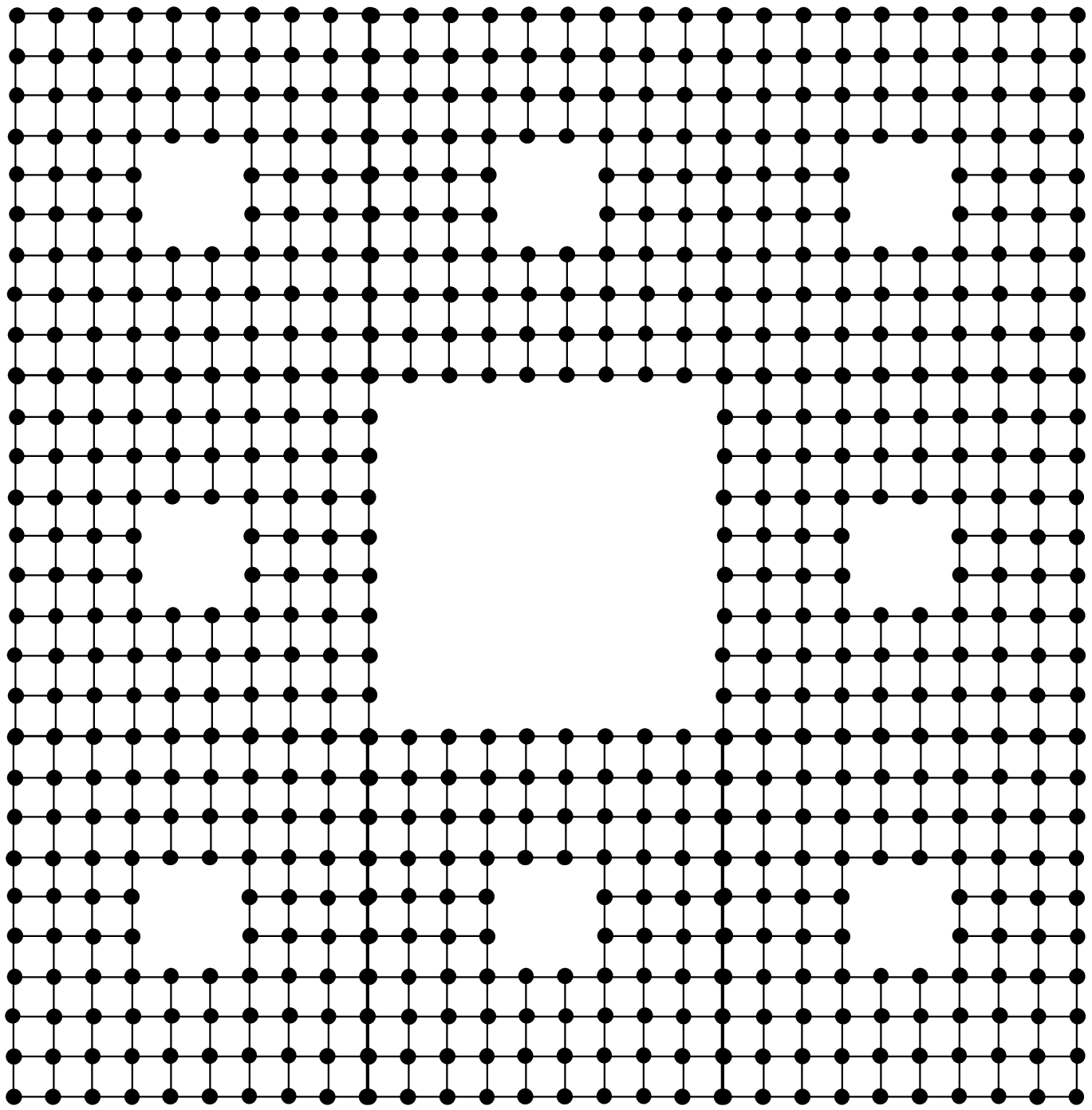}
  \end{center}
 \end{minipage} 
 \end{tabular}  
      \caption{ A sequence of the Sierpinski carpet graphs. }
\end{figure}

\begin{Example}[Sierpinski carpet graph]  \label{Ex:Carpet10}   
Fixing integers $L \ge 2$ and $K \in [L,L^d]$, partition the $d$-dimensional unit cube $H_0 = [0,1]^d$ into the  
collection  $Q := \{ \prod_{i=1}^d [ \frac{ (k_i -1) }{L} , \frac{ k_i  }{L} ]  \mid 1 \le k_i \le L \text{ for all } i \in \{ 1,2, \ldots , d \} \}$ of 
$L^d$ sub-cubes. Then fixing $L$-similitudes 
$\{ \psi_i, 1 \le i \le K\}$ of $H_0$ onto 
mutually distinct 
elements of $Q$, 
such that $\psi_1 (x): = L^{-1} x$, there exists a unique non-empty 
compact $F \subset H_0$ such that $F = \bigcup_{i=1}^K \psi_i (F)$. 
We call $F$ the generalized Sierpinski carpet if 
 the following four
conditions hold:

\smallskip
\noindent
(a) (Symmetry)  $H_1 := \bigcup_{i=1}^K \psi_i (H_0)$ is preserved by all  isometries of $H_0$.

\smallskip
\noindent
(b) (Connectedness) $\Int (H_1)$ is connected, and contains a path connecting the hyperplanes $\{ x_1 =0 \}$ and $\{ x_1 = 1\}$.

\smallskip
\noindent
(c) (Non-diagonality) 
If $\Int (H_1 \cap B)$ is nonempty for some $d$-dimensional cube 
$B \subset H_0$ which is the union of $2^d$ distinct elements of $Q$,
then $\Int (H_1 \cap B)$ is a connected set.

\smallskip
\noindent
(d) (Borders included) $H_1$ contains the line segment 
$\{ (x_1,0,\ldots,0) \mid 0 \le x_1 \le 1 \}$.     

\smallskip
\noindent

For a generalized Sierpinski carpet, let $V^{(0)}$ and $E^{(0)}$ denote the
the $2^d$ corners of $H_0$ and $d 2^{d-1}$ edges on the boundary of $H_0$ respectively, with 
\begin{align*}
V(G^{(N)}) := \bigcup_{i_1, i_2, \ldots, i_N=1}^K L^N \psi_{i_1, i_2, \ldots, i_N } (V^{(0)}), \qquad \qquad
                     V(G) := \bigcup_{N \ge 1} V(G^{(N)}) \,.    \\
E (G^{(N)}) :=  \bigcup_{i_1, i_2, \ldots, i_N=1}^K L^N \psi_{i_1, i_2, \ldots, i_N } (E^{(0)}), \qquad \qquad
                     E(G) := \bigcup_{N \ge 1} E(G^{(N)}) \,.    \\
  \end{align*}
Once again, it is easy to check that if Assumption \ref{Ass:Weight}(a) holds 
for weight $\mu^{(N)}$ on $G^{(N)}$, then
such $\mu^{(N)}$ satisfies also Assumptions \ref{Ass:Weight}(b) 
and \ref{Ass:Weight}(c) for $d_f = \log K/\log L$.
 \end{Example}  

We prove in Section \ref{Subsec:pfSC} the following.
\begin{Proposition}\label{thm:SC12}
For any generalized Sierpinski carpet, the 
weighted graphs $\{ (G^{(N)}, \mu^{(N)}) \}_{N \ge 0}$ 
of Example \ref{Ex:Carpet10} further satisfy Assumption \ref{Ass:UPHI}
for some finite $d_w = \log (\rho K)/\log L$. 
\end{Proposition}

Whereas directly verifying Assumption \ref{Ass:UPHI} is often difficult,
as shown in Section \ref{Sec:pfSGSC}, certain conditions from 
the research on \abbr{sub-GHKE} are equivalent to \abbr{PHI} and
more robust. Indeed those equivalent conditions are key 
to our proof of Propositions \ref{thm:SG12} and \ref{thm:SC12}.

In the context of Example \ref{Ex:Carpet10}, for carpets 
with central block of size $b^d$ removed (so 
$K=L^d - b^d$), for some $1 \le b \le L-1$, 
one always have $\rho > 1$ when $d=2$
(see \cite[\abbr{lhs} of (5.9)]{BB99}), 
hence by Theorem \ref{Thm:Main}(a) no cutoff
for the corresponding lamplighter chain. In contrast, 
from \cite[\abbr{rhs} of (5.9)]{BB99} we know that 
$\rho <1$ for high-dimensional carpets of small central 
hole (specifically, whenever $b^{d-1} < L^{d-1} - L$), 
so by Theorem \ref{Thm:Main}(b)
the corresponding 
lamplighter chains then 
admit cutoff at $a_N = \frac{1}{2} T_{\rm{cov}} (G^{(N)})$.

\subsection{Stability of heat kernel estimates and parabolic Harnack inequality}    \label{Sec:pfSGSC}

We recall here various stability results for Heat Kernel Estimates (\abbr{hke})
and Parabolic Harnack Inequalities (\abbr{PHI}), in case of 
a countably infinite weighted graph $(G,\mu)$. To this end, we assume 
\begin{itemize}
\item  \emph{Uniform ellipticity:} 
$\ce^{-1} \le \mu_{xy} \le \ce$  for some $\ce<\infty$ and 
            all $x y \in E(G)$,
\item  \emph{$p_0$-condition:} $\frac{\mu_{xy} }{\mu_x}  \ge p_0$ 
for some $p_0>0$ and all $x y \in E(G)$, 
    \end{itemize}  
and recall few relevant properties of such $(G,\mu)$.
\begin{Definition} Consider the following properties 
for $d_w \ge 2$ and $d_f \ge 1$:
 \begin{itemize}
       \item  (VD) There exists $\CD < \infty$  such that 
                    \begin{align*}
                                 V(x, 2r)   \le \CD V(x,r)  \qquad  \text{ for all $x \in V(G)$ and $r \ge 1$.}
                    \end{align*}   

       \item  (V($d_f$)) There exists $\CV < \infty$ such that 
                     \begin{align*}
\CV^{-1} r^{d_f}   \le  V(x,r)   \le \CV r^{d_f} \qquad  \text{ for all $x \in V(G)$ and $r \ge 1$.}
                    \end{align*}   

      \item  (CS($d_w$)) There exist $\theta >0$, $\CCS < \infty$ 
      and for each $z_0 \in V(G)$, $R \ge 1$ there exists 
      a cut-off function 
$\psi = \psi_{z_0,R} : V(G) \to \mathbb{R}$ such that:   

\smallskip\noindent
(a) $\psi (x) \ge 1$ when $d(x,z_0) \le R/2$, while $\psi (x) \equiv 0$ 
when $d(x,z_0) > R$, 
 
\smallskip\noindent
(b) $| \psi (x) - \psi (y) | \le \CCS  \left(d(x,y)/R\right)^{\theta}$,

\smallskip\noindent
(c) for any $z \in V(G)$, $f: B(z,2s) \to \mathbb{R}$ and $1 \le s \le R$ 
\begin{align*}
\sum_{x \in B(z,s)} f(x)^2 & 
\sum_{y \in V(G)} | \psi (x) - \psi (y) |^2 \mu_{xy} \\
&  \le  \CCS^2   \Big(  \frac{s}{R}  \Big)^{2\theta} 
                                        \Big(  \sum_{x,y \in B(z,2s)} |f(x) - f(y)|^2 \mu_{xy}  +  s^{-d_w}  \sum_{y \in B(z,2s)} f(y)^2  \mu_y \Big) .  
                                \end{align*}

     \item (PI($d_w$))  There exists $\CPI < \infty$ such that 
                \begin{align*}
                              \sum_{ x \in B(z,R) }  (f(x) - \bar{f}_{B(z,R)}  )^2  \mu_x  \le \CPI \, R^{d_w}  \sum_{x,y \in B(z,2R)} (f(x) - f(y))^2 \mu_{xy}
                \end{align*}
    for all $R \ge 1$, $x \in V(G)$ and $f : V(G) \to \mathbb{R}$, where $\bar{f}_{B(z,R)}  =\frac{1}{ V(z,R) } \sum_{x \in B(z,R)}  f(x) \mu_x$.

      \item (HKE($d_w$))  There exists $\CHK < \infty$ such that 
                 \begin{align*}  
                             p_t (x,y)  \le    \frac{ \CHK }{ V(x, t^{1 / d_w} ) } \exp \Big[  - \frac{1}{\CHK} \Big(  \frac{ d (x,y)^{d_w} }{ t }  \Big)^{1/(d_w -1) }  \Big] 
                  \end{align*}
               for all $x,y \in V(G)$ and $t \ge 0$, whereas 
                  \begin{align*}
                           p_t (x,y)  +  p_{t+1}  (x,y)  \ge    \frac{1}{\CHK V(x, t^{1 / d_w} )} \exp \Big[  - \CHK \Big(  \frac{ d (x,y)^{d_w} }{ t }  \Big)^{1/(d_w -1) }  \Big]
                  \end{align*}
                 for all $x,y \in V(G)$ and $t \ge d(x,y)$.  
  
     \item  (PHI($d_w$)) There exists $\CPHI < \infty$ such that if 
     $u : [0, \infty) \times V(G) \to [0,\infty )$ satisfies 
              \begin{align*}
                      u (t+1, x)  - u (t,x)   = (P - I) u(t,x),  \qquad 
                      \forall \, (t,x) \in [0,4T] \times B (x_0,2R)  
              \end{align*}
    for some $x_0 \in V(G)$, $T \ge 2R$ with $T \asymp R^{d_w}$, then, 
    for such $x_0,T,R$,
              \begin{align*}   
                       \max_{ \substack{ z \in B(x_0, R ) \\ s \in [T, 2T]} }  u(s,z )  
                             \le  \CPHI   \min_{ \substack{ z \in B (x_0,R ) \\ s \in [3T, 4T]} } \{ u(s,z) + u(s+1,z) \} .  
              \end{align*}  
 
  \end{itemize}
\end{Definition}

\begin{Theorem}[{\cite[Theorems 1.2, 1.5]{BB04}}] $~$  \label{Thm:Equiv}  
The following are equivalent for any uniformly elliptic, 
countably infinite 
$(G, \mu)$ satisfying the $p_0$-condition:  
          \begin{enumerate}  \renewcommand{\labelenumi}{(\alph{enumi})}
              \item  (VD), (PI($d_w$)) and (CS($d_w$)). 
              \item  (HKE($d_w$)).
              \item  (PHI($d_w$)).  
          \end{enumerate}          
  \end{Theorem} 
Note that in each implication of Theorem \ref{Thm:Equiv} the resulting
values of $(\CD, \CPI,\theta,\CCS)$, $\CHK$ and $\CPHI$ 
depend only on $p_0$, $\ce$, $d_w$ and the assumed constants.
For example, in (PHI($d_w$))$ \Rightarrow $ (HKE($d_w$)), 
the value of $\CHK$ depends only on $\CPHI$, $p_0$, $\ce$ and $d_w$.  

\medskip
The stability of such equivalence involves the 
following notion of rough isometry (see \cite[Definition. 5.9]{HK04}). 
\begin{Definition}
Weighted graphs $(G^{(1)} , \mu^{(1)})$ and $(G^{(2)}, \mu^{(2)})$ are rough isometric if there exist $\CQI<\infty$ and 
a map $T : V^{(1)} \to V^{(2)}$ such that 
        \begin{align*}
         \CQI^{-1} \, d^{(1)} (x,y) - \CQI \le d^{(2)} (T(x),T(y))  & \le \CQI \, d^{(1)} (x,y) + \CQI \,,  \quad  & \forall x,y \in V^{(1)},   \\
                   d^{(2)} (x^{\prime}, T (V^{(1)})) & \le \CQI,    & 
                   \forall x^{\prime} \in V^{(2)},    \\  
                   \CQI^{-1} \, \mu_x^{(1)}   \le \mu_{T(x)}^{(2)}  & \le  \CQI \, \mu_x^{(1)},    & \forall x \in V^{(1)},  
        \end{align*}
where $d^{(i)}(\cdot,\cdot)$ and $V^{(i)}$ denote the graph distance and vertex set 
of $G^{(i)}$, $i=1,2$, respectively. Similarly, weighted graphs 
$\{ (G^{(N)}, \mu^{(N)} ) \}_{N}$  are uniformly rough isometric 
to a fixed, weighted graph $(G, \mu)$ if each 
$(G^{(N)}, \mu^{(N)})$ is rough isometric to $(G, \mu)$ 
for some $\CQI < \infty$ which does not depend on $N$.  
\end{Definition}  

Recall \cite[Lemma 5.10]{HK04}, that rough isometry is an 
equivalence relation. Further, (VD), (PI($d_w$)) and (CS($d_w$)) 
are stable under rough isometry. That is,
\begin{Theorem}[{\cite[Proposition 5.15]{HK04}}]   \label{Thm:Stab}
Suppose $(G^{(1)} , \mu^{(1)})$ and $(G^{(2)} , \mu^{(2)})$ 
have the $p_0$-condition and are rough isometric with constant $\CQI$. 
If $(G^{(1)} , \mu^{(1)})$ satisfies  (VD), (PI($d_w$)), (CS($d_w$)) 
with constants $(\CD,\CPI,\theta,\CCS)$, 
then so does $(G^{(2)},\mu^{(2)})$ with constants which depend only on 
     $(\CD,\CPI,\theta,\CCS)$, $d_w$, $p_0$ and $\CQI$. 
\end{Theorem}

Combining Theorems \ref{Thm:Equiv} and \ref{Thm:Stab} 
we have the following useful corollary. 
 \begin{Corollary}  \label{Cor:Stab20}
Suppose uniformly elliptic weighted graphs $\{ (G^{(N)}, \mu^{(N)} ) \}_{N}$  
satisfy the $p_0$-condition and are uniformly rough isometric to 
some countably infinite uniformly elliptic  
$(G, \mu)$ that also has the $p_0$-condition. If
$(G, \mu)$ further satisfies (PHI($d_w$)), then so do 
$\{ (G^{(N)}, \mu^{(N)} ) \}_{N}$ with finite constant $\CPHI^{\prime}$ 
which is independent of $N$. 
  \end{Corollary}

\subsection{Proof of Propositions \ref{thm:SG12} and \ref{thm:SC12}}  \label{Subsec:pfSC}
 
\begin{proof}[Proof of Proposition \ref{thm:SG12}.] 
Recall that for random walks on the Sierpinski gasket, namely
$\mu_{xy} \equiv 1$ and the limit graph $G$ of Example \ref{Ex:Gasket10} 
(or its $d$-dimensional analog, $d \ge 3$),  
Jones \cite[Theorems 17,18]{Jones96} established (HKE($d_w$)),
 which by Theorem \ref{Thm:Equiv} implies that 
such $(G,\mu)$ must also satisfy (PHI($d_w$)).

\begin{figure}[htbp]
  \begin{tabular}{cc}
       \begin{minipage}{0.5\hsize}
            \begin{center}
                    \includegraphics[width=60mm]{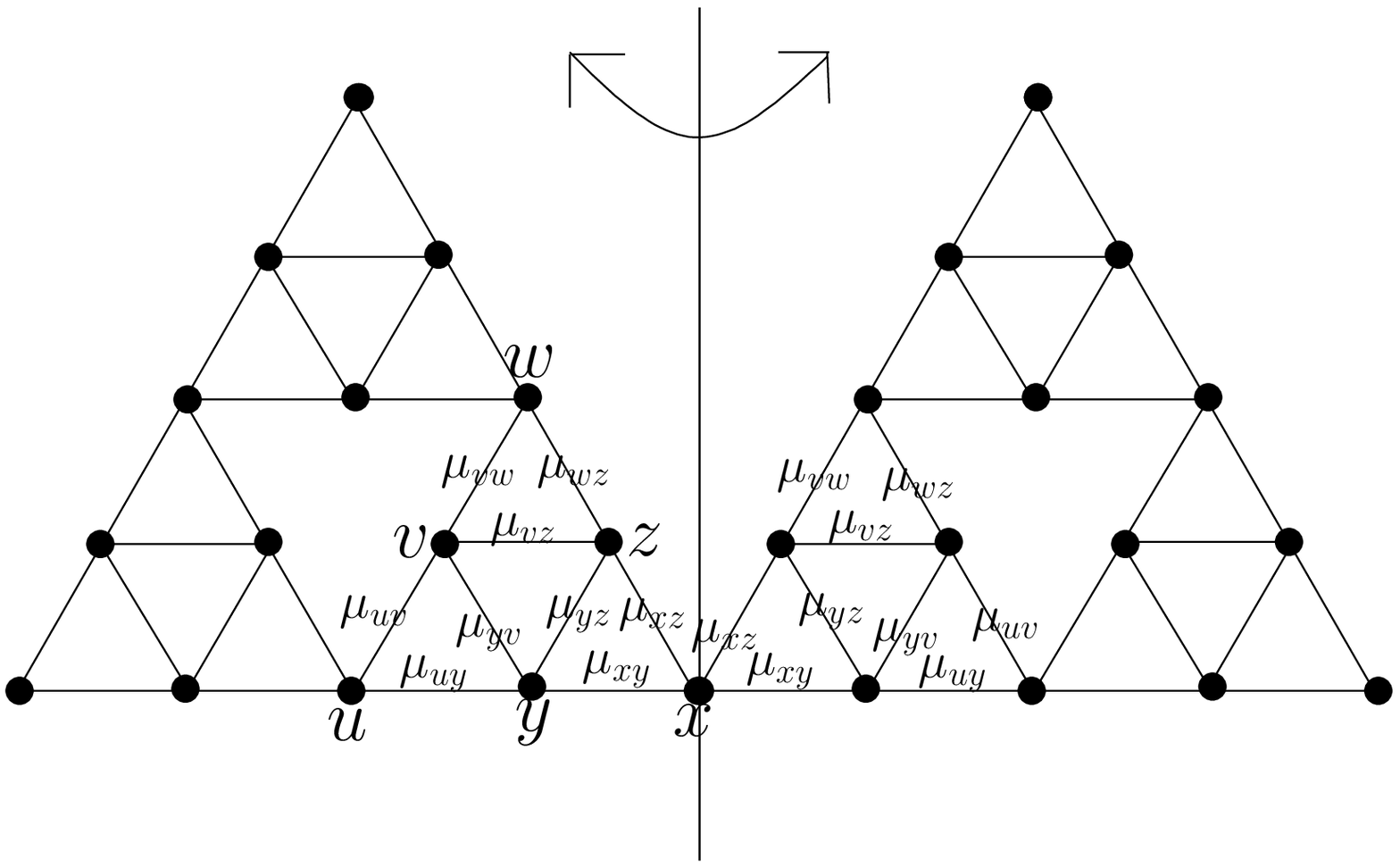}
            \end{center}
       \end{minipage} 
  \begin{minipage}{0.5\hsize}
           \begin{center}
                      \includegraphics[width=40mm]{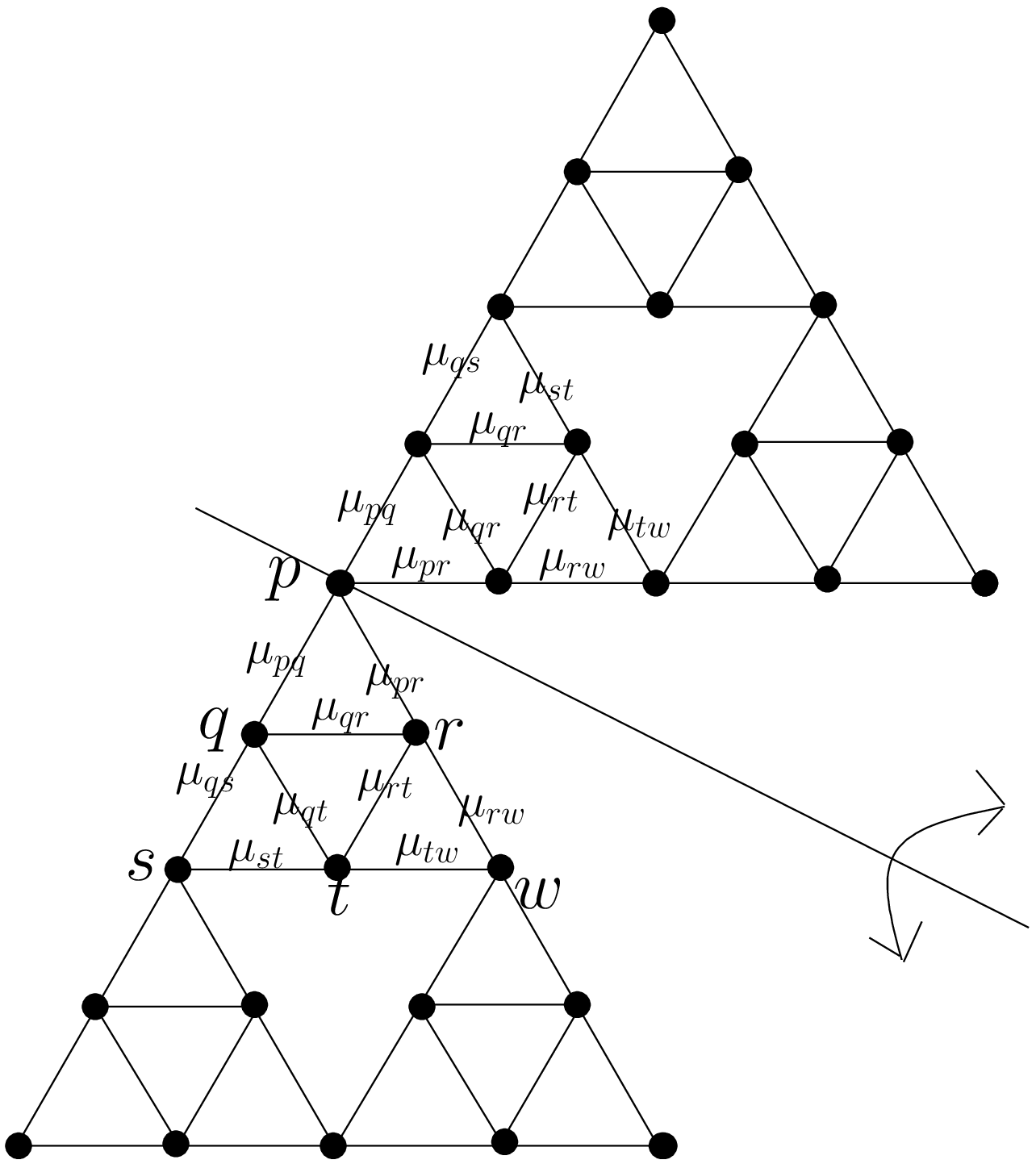}
           \end{center}
   \end{minipage} 
 \end{tabular}  
      \caption{ The construction of a weighted graph $(G^{(N+1)}, \mu^{(N+1)} )$ for a given  $(G^{(N)}, \mu^{(N)} )$. } 
       \label{Fig:Gasket30}
\end{figure}
                    
Proceeding to construct for each $N \ge 1$ a new 
weighted graph $(G, \mu^{\prime (N)})$, recall that 
$G^{(N+1)}$ consists of three copies $G^{(N,i)}$ of $G^{(N)}$, 
with $2^N x_i \in G^{(N,i)}$ for $i=0,1,2$. Note that 
$G^{(N,0)}=G^{(N)}$ whereas each 
$G^{(N,i)}$, $i=1,2$ is the reflection of $G^{(N,0)}$  
across a certain line $\ell^{N,i}$. Reflecting 
the weight $\mu^{(N,0)} := \mu^{(N)}$ on $G^{(N,0)}$, 
across $\ell^{N,i}$ yields weights $\mu^{(N,i)}$ on 
$G^{(N,i)}$, $i=1,2$ (see Figure \ref{Fig:Gasket30}). 
With $\{\mu^{(N,i)}, i=0,1,2\}$ forming a new 
weight on $G^{(N+1)} \subset G$, we thus set 
\begin{align*}   
\mu^{\prime (N)}_{xy}   : = 
                                 \begin{cases}
                                         \mu^{(N,i)}_{xy},  &  \text{ if $xy \in 
                                         E(G^{(N+1)})$,}   \\
                                          1 ,                 &  \text{otherwise}. 
                                  \end{cases}
                \end{align*}
Fixing a solution $u^{(N)} : [0, \infty) \times V(G^{(N)}) \to [0, \infty)$ 
of the heat equation \eqref{eq:he-GN} on the time-space 
cylinder of center $y_0 \in V(G^{(N)})$ and size
$2 R \le T \asymp R^{d_w}$, 
$R \le \frac{1}{4} R_N$,  
we extend $u^{(N)}(t,\cdot)$ to the non-negative function on $V(G)$ 
       \begin{align}    \label{eq:SGRef}
                       \tilde{u}^{(N)}  (t,x) := 
             \begin{cases}  
                   u^{(N)} (t,x),    & \text{ if $x \in V(G^{(N)})$},  \\   
                   u^{(N)} (t, x^{\prime}),  &  \text{ if $x \not\in V(G^{(N)})$ and $x^{\prime}$ are symmetric \abbr{w.r.t.} $\ell^{N,1}$ or $\ell^{N,2}$}, \\
                    0,    &  \text{ otherwise}. 
             \end{cases}
       \end{align}  
Having $R \le \frac{1}{4} R_N$
guarantees that $B_G (y_0,2R) \subseteq G^{(N+1)}$,
hence from our construction of $\mu^{\prime (N)}$ it follows that
$\tilde{u}^{(N)}(t,x)$ satisfy the heat equation 
corresponding to $(G,\mu^{\prime (N)})$ on the time-space cylinder
defined by $(y_0,R,T)$. Since $G$ has uniformly bounded degrees,
the weighted graphs $\{(G,\mu^{\prime (N)})\}_N$ satisfy a $p_0'$-condition 
(for some $p_0'>0$ independent of $N$). Further, $\{(G,\mu^{\prime (N)})\}_N$
are uniformly rough isometric 
to $(G, \mu)$ (thanks to the uniform ellipticity of $\mu^{(N)}$).
Hence, by Corollary \ref{Cor:Stab20}, for some $\CPHI' < \infty$,
which does not depend on $N$, nor on the specific 
choice of $y_0$, $R$ and $T$,
          \begin{align}  \label{eq:SGPHIN}
                   \max_{ \substack{ t \in [T,2T] \\ y \in B_G(y_0, R) } } \tilde{u}^{(N)} (t,y)  
                         \le \CPHI' \min_{ \substack{ t \in [3T,4T] \\ y \in B_G(y_0, R) } } \{ \tilde{u}^{(N)} (t,y) + \tilde{u}^{(N)} (t+1, y) \} .  
          \end{align} 
Since $\tilde{u}^{(N)}$ of \eqref{eq:SGRef} coincides with $u^{(N)}$ on 
$B^{(N)}(y_0,R) \subseteq B_G(y_0,R)$, replacing $\tilde{u}^{(N)}$ and
$B_G(y_0,R)$ in \eqref{eq:SGPHIN} by $u^{(N)}$ and $B^{(N)}(y_0,R)$, 
respectively, may only decrease its \abbr{lhs} and increase its \abbr{rhs}.
That is, \eqref{eq:SGPHIN} applies also for $u^{(N)}(\cdot,\cdot)$
and $B^{(N)}(y_0,R)$. This holds for all $N$ and any of the
preceding choices of $y_0,R,T$, yielding Assumption \ref{Ass:UPHI},
as stated.  
\end{proof}

\begin{proof}[Proof of Proposition \ref{thm:SC12}.]
Consider the random walk, namely $\mu_{xy} \equiv 1$, on 
a limiting graph $G$ that corresponds to a generalized 
Sierpinski carpet, as in Example \ref{Ex:Carpet10}. Clearly,
$(G,\mu)$ is uniformly elliptic and of uniformly bounded degrees
(so $p_0$-condition holds as well). Further, such random walk 
has properties (V($d_f$)) and (HKE($d_w$)), with 
$d_f = \log K / \log L \ge 1$ and 
$d_w = \log (\rho K)/\log L$ 
(see \cite{BB97}). In particular, by Theorem \ref{Thm:Equiv}
$(G,\mu)$ satisfies (PHI($d_w$)). With $G^{(N+1)}$  
consisting of $K$ copies of $G^{(N)}$, we 
extend the given weight $\mu^{(N)}$ on 
$G^{(N)}$ to a weight $\mu^{\prime (N)}$ on $G$. Specifically, 
the weight on the edges of the reflected part of $G^{(N)}$, 
as in Figure \ref{Fig:Carpet-Ref}, is $\mu^{\prime (N)}_{e} = K_e \mu^{(N)}_{e^{\prime}}$, where 
$K_{e} \in [1,K]$ is the number of overlaps of $e$, and 
$e^{\prime}$ is the edge which moves to $e$ by the reflection
(so in Figure \ref{Fig:Carpet-Ref}, we set $\mu^{\prime,(N)}_{e} = 2 \mu^{(N)}_e$ for each edge $e$ lying on a reflection axis). 
Taking $\mu^{\prime (N)}_{e} \equiv 1$ for all other $e \in E(G)$, 
the graphs $\{ (G, \mu^{\prime (N)}) \}_{N}$ are uniformly elliptic,
satisfy a $p_0'$-condition (for some $p_0'>0$ independent of $N$),
and are uniformly rough isometric to $(G, \mu)$. Thus, by 
Corollary \ref{Cor:Stab20} the
(PHI($d_w$)) holds for $\{ (G, \mu^{\prime (N)}) \}_{N}$ with a 
constant $\CPHI^{\prime}$ which does not depend on $N$. 
\begin{figure}[htbp]
  \begin{center}  
   \includegraphics[width=60mm]{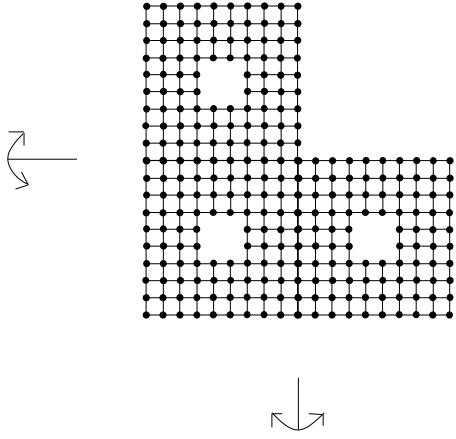}
  \end{center} 
      \caption{ An example of the reflection}   \label{Fig:Carpet-Ref}
\end{figure}
Fixing center $y_0 \in V(G^{(N)})$ and size parameters
$2 R \le T \asymp R^{d_w}$, 
$R \le \frac{1}{4} R_N$, we extend 
any given solution $u^{(N)} : [0,\infty) \times V(G^{(N)}) \to [0, \infty)$ 
of the heat equation \eqref{eq:he-GN} on the corresponding 
time-space cylinder, to the non-negative $\tilde{u}^{(N)} :
[0,\infty) \times V(G) \to [0,\infty)$, symmetrically 
along reflections, analogously to \eqref{eq:SGRef}. 
Since $R \le \frac{1}{4} \Diam \{G^{(N)}\}$ all edges
of $B_G (y_0,2R)$ not in $G^{(N)}$ are among those 
reflected to $G^{(N)}$, with our construction of $\mu^{\prime (N)}$ 
guaranteeing that $\tilde{u}^{(N)}(\cdot,\cdot)$ satisfy the heat equation 
on the corresponding time-space cylinder of $(G,\mu^{\prime (N)})$. 
Thanks to the (PHI($d_w$)) for $\{ (G, \mu^{\prime (N)}) \}_{N}$, 
we have \eqref{eq:SGPHIN}, and since
$\tilde{u}^{(N)}$ coincides with $u^{(N)}$ on 
$B^{(N)}(y_0,R) \subseteq B_G(y_0,R)$, 
the same applies when replacing $\tilde{u}^{(N)}$ and
$B_G(y_0,R)$ by $u^{(N)}$ and $B^{(N)}(y_0,R)$, 
respectively. As in our proof of Proposition \ref{thm:SG12}, this holds 
for all relevant values of $N$, $y_0$, $R$ and $T$, 
thereby establishing Assumption \ref{Ass:UPHI}.  
\end{proof}

\section{Random walk consequences of Assumptions \ref{Ass:Weight} and \ref{Ass:UPHI}}   \label{Sec:Pre}

We summarize here those consequences of Assumptions \ref{Ass:Weight} and \ref{Ass:UPHI} we need for Theorem \ref{Thm:Main}, 
starting with \abbr{sub-GHKE}, an upper bound on the 
uniform mixing times and a covering statement which are 
applicable for all values of $(d_f,d_w)$. Then, focusing in Section 
\ref{Subsec:Rec} on the case $d_f < d_w$, 
we control $R_{\rm{eff}}^{(N)} (\cdot, \cdot)$ and
relate it to $T_N$ of \eqref{eq:Time}, complemented in 
Section \ref{Subsec:Trans} by upper bounds on the Green functions, 
in case $d_f > d_w$. We provide only proof outlines since most 
of these results, and their proofs, are pretty standard.

\medskip\noindent
Our first result is the uniform \abbr{sub-GHKE} one has on
$\{(G^{(N)},\mu^{(N)})\}_{N \ge 1}$, up to time of order $T_N$.

\begin{Proposition}  \label{Prop:HKEN} 
Under Assumptions \ref{Ass:Weight} and \ref{Ass:UPHI}, for any
$\eta<\infty$, there exist 
$\cHK = \cHK (\eta ) < \infty$, such that for all $N$, any
$x,y \in V(G^{(N)})$ and $t \le \eta T_N$,
\begin{align}\label{eq:HKU}
                      p_t^{(N)} (x,y)  \le  \frac{\cHK}{t^{d_f/d_w}} 
\exp \Big[  -  \frac{1}{\cHK} \Big(  \frac{ d^{(N)} (x,y)^{d_w}}{t}  \Big)^{1/(d_w -1) }  
\Big] \,.
\end{align}
Further, for all $N$, any $x,y \in V(G^{(N)})$ and
$d^{(N)} (x,y) \le t \le \eta T_N$,
 \begin{align}\label{eq:HKL} 
 p_{t}^{(N)} (x,y)  + p_{t+1}^{(N)} (x,y) \ge  \frac{1}{\cHK t^{d_f/d_w}}   
 \exp \Big[  - \cHK \Big(  \frac{ d^{(N)} (x,y)^{d_w} }{ t }  \Big)^{1/(d_w -1) }  \Big] \,.  
 \end{align}
\end{Proposition}

\begin{proof} (Sketch:) This is a finite graph analogue of 
(PHI($d_w$))$ \Rightarrow $ (HKE($d_w$)) of Theorem \ref{Thm:Equiv}, which  
is standard for a 
countably infinite weighted graph (see 
\cite[Theorem 3.1, (ii) $\Rightarrow$ (i)]{GT02}). Such implication 
holds also for metric measure space with a 
local regular Dirichlet form, as 
\cite[Theorem 3.2 $(c^{\prime}) \Rightarrow (a'')$]{BGK12}, 
and we sketch below how to adapt the latter proof, 
specifically 
\cite[Sections 4.3 and 5]{BGK12}, to the finite graph setting. 
First note that for $t \le \eta T_N$ the derivation of the 
(near-)diagonal upper-bound \eqref{eq:HKU} (without the 
exponential term), follows as in the proof of \cite[Proposition 7.1]{Telcs01}. 
Setting $p_t^{(N,x,R)}$ for the heat kernel of the process 
killed upon exiting $B^{(N)}(x,R)$, upon 
adapting the arguments in \cite[Section 4.3.4]{BGK12}, one thereby
establishes the corresponding (near)-diagonal lower bound,  analogous to 
\cite[(4.63)]{BGK12}. Namely, showing that for some $\cPHI' \in (0,1)$ 
and $\cHK'=\cHK'(\eta')$ finite, any $\eta'<\infty$, 
all $N \ge 1$, $x \in V(G^{(N)})$ and $R \le \cPHI' \, R_N$,
if $\cHK' \, d^{(N)}(x,y) \le t^{1/d_w} \le \eta' R$, then 
\begin{align}  \label{eq:KHKE}
p_t^{(N,x,R)} (x,y) + p_{t+1}^{(N,x,R)}(x,y) \ge  \frac{1}{\cHK' t^{d_f /d_w}}  \,.
\end{align} 
Combining \eqref{eq:KHKE} and the (near-)diagonal upper bound, one then deduces
\eqref{eq:HKU} as done in \cite[Sections 4.3.5-4.3.6]{BGK12}. Similarly, 
by adapting the proof of \cite[Proposition 5.2(i) and (iii)]{BGK12}, 
the near-diagonal lower bound \eqref{eq:KHKE} yields the full lower-bound 
of \eqref{eq:HKL}. Since all these arguments involve only 
$\eta$ and the constants from Assumptions \ref{Ass:Weight}-\ref{Ass:UPHI},
we can indeed choose the constant $\cHK(\eta)$ 
in \eqref{eq:HKU}-\eqref{eq:HKL} independently of $N$.  
\end{proof}

Proposition \ref{Prop:HKEN} has the following immediate consequence.
\begin{Corollary}  \label{Cor:HKEN} 
Under Assumptions \ref{Ass:Weight} and \ref{Ass:UPHI} there 
exist $R_0$ and $c_2$ finite, such that 
for any $N \ge 1$, $x\in V(G^{(N)})$ and
$R_0 \le r \le R_N$
\begin{gather*}
P_x \big( \max_{0 \le j \le t} d^{(N)}(x,X_j^{(N)}) \le r \big)  
\ge c_2^{-1} \exp(-c_2 t/r^{d_w}) \,. 
\end{gather*} 
\end{Corollary}
\begin{proof} Using the same arguments as in the proof of \cite[Proposition 3.3]{KN16}, 
from \eqref{eq:HKU} and \eqref{eq:HKL} we get the finite graph analogs of 
\cite[Lemma 3.1]{KN16} and \cite[Lemma 3.4]{KN16}, respectively. Combining 
these bounds and the Markov property, as done in \cite[Lemma 3.5]{KN16}, 
results with the stated bound for $k [r^{d_w}] \le t < (k+1) [r^{d_w}]$.
All steps of the proof involve only our universal constants 
$\ce$, $\cv$, $p_0$, $\CPHI$, $\cPHI$, $\cHK$ and with 
$X_j^{(N)}$ confined to certain balls, having our \abbr{sub-HKE} 
restricted to $t \le \eta T_N$ is immaterial here.
\end{proof}

Another consequence of \eqref{eq:HKU} is 
the following upper bound on uniform mixing times.   
\begin{Proposition}  \label{Prop:MTE}
Suppose Assumptions \ref{Ass:Weight} and \ref{Ass:UPHI} hold.  
Then, for the invariant measures $\pi^{(N)}(\cdot)$ of \eqref{eq:InvUnder},
some finite $c(\cdot)$, all $N \ge 1$ and $\epsilon >0$, 
            \begin{align}\label{eq:Umix}
                     T_{\rm{mix}}^U (\epsilon; G^{(N)} ) 
:= \min  \Big\{  t \ge 0 \, \big|  \max_{x,y \in V(G^{(N)})} \Big| 
\frac{ P_x (\tilde{X}_t =y; G^{(N)})}{\pi^{(N)}(y)} - 1 \Big| \le \epsilon  \Big\}   
 \le c(\epsilon) T_N \, .  
\end{align} 
\end{Proposition}  

For the proof of Proposition \ref{Prop:MTE}, 
consider the normalized Dirichlet forms of 
$\tilde X^{(N)}$ and $X^{(N)}$,
\begin{align*}  
\mathcal{E}_{\rm{norm} }^{(N)} (f, f)   &  := -  \langle f, (P^{(N)} - I) f \rangle_{ \pi^{(N)} },      \\    
\tilde{ \mathcal{E} }_{\rm{norm} }^{(N)} (f, f)   & := -  \langle f, (\tilde{P}^{(N)} - I) f \rangle_{ \pi^{(N)} } = \frac{1}{2}   \mathcal{E}_{\rm{norm} }^{(N)} (f, f) \,.     
\end{align*} 
Let
${\mathcal H}_0^{+} (S)  
:= \{ f: V(G^{(N)}) \to [0, \infty ) \mid f$ not a constant function, 
$\Supp \{ f \} \subseteq S \}$ for 
$S \subseteq V(G^{(N)})$ and define the spectral quantities 
\begin{align*}
\lambda^{(N)} (S)  :=  \inf \left\{   \frac{ \mathcal{E}_{\rm{norm} }^{(N)} (f, f) }{  \Var^{\pi^{(N)}} (f)  } 
                    \Biggm|  
                    f \in {\mathcal H}_0^{+} (S)  \right\},   \;\; 
\tilde{ \lambda }^{(N)} (S)  :=  \inf \left\{   \frac{ \tilde{ \mathcal{E} }_{\rm{norm} }^{(N)} (f, f)  }{  \Var^{\pi^{(N)}} (f)  }  
\Biggm| f \in {\mathcal H}_0^{+} (S)  \right\} = 
\frac{1}{2}  \lambda^{(N)} (S) \,.
\end{align*}  
Recall the following upper bound on uniform mixing times 
in terms of the corresponding spectral profile. 
\begin{Lemma}[{\cite[Corollary 2.1]{GMT06}}]   \label{Lem:MTSP}  
For 
$r \ge \pi_{\ast}^{(N)} :=  \inf_{x \in V(G^{(N)})} \{ \pi^{(N)} (x)\}$, let
\begin{align*}
                     \tilde{ \Lambda}^{(N)} (r) :=   \inf
                      \left\{  \tilde{\lambda}^{(N)} (S) \,|\,
\pi^{(N)} (S) \le r                      
                       \right\} \,.
\end{align*}
Then, for any $\epsilon >0$ and all $N$, 
   \begin{align}\label{eq:Umix-bd}
           T_{ \rm{mix} }^U (\epsilon ; G^{(N)} ) \le  
           \int_{4 \pi_{\ast}^{(N)} }^{4/\epsilon} \;\; \frac{4 \, dr}{r \tilde{\Lambda}^{(N)} (r) } \,. 
   \end{align} 
\end{Lemma}
\noindent
Our next lemma controls the spectral profiles on the 
\abbr{rhs} 
of \eqref{eq:Umix-bd} 
en-route to Proposition \ref{Prop:MTE}.
\begin{Lemma}[Faber-Krahn inequality]  \label{Lem:FKN}
For any $N$ and $S \subseteq V(G^{(N)})$ let
\begin{align}\label{eq:Spec}  
\lambda_1^{(N)} (S)  :=   \inf \Big\{  \, \frac{ \mathcal{E}^{(N)} (f,f) }{ \| f \|_{L^2 (\mu^{(N)} )  }  }  
            \,\Big| \,
              f \in {\mathcal H}_0 (S)  \Big\}\,,  
\end{align}
where ${\mathcal H}_0 (S) 
:= \{ f : V(G^{(N)}) \to \mathbb{R} \mid \Supp \{f\} \subseteq S \}$.
If Assumptions \ref{Ass:Weight} and \ref{Ass:UPHI} hold, then for 
some $\cFK>0$ and all $N$, 
\begin{align}  \label{Prop:FKN02}
              \lambda_1^{(N)} (S)  \ge   \cFK \, \mu^{(N)} (S)^{-d_w / d_f}   
               \qquad   \forall S \subseteq V(G^{(N)})\,.        
\end{align}
\end{Lemma}
\begin{proof}(Sketch) For countably infinite $(G,\mu)$ satisfying 
the $p_0$-condition, such Faber-Krahn inequality is a standard 
consequence of (V($d_f$)) and the on-diagonal (HKE($d_w$)) upper bound. 
Indeed, its proof in \cite[Theorem 5.4]{CG98}, while written for $d_w=2$, 
is easily adapted to any $d_w > 0$, upon suitably adjusting various exponents
(e.g. taking $\nu=d_w/d_f$ and 
$r=t^{1/d_w}$, c.f. the discussion in 
\cite[Proposition 5.1]{GT01}). To get \eqref{Prop:FKN02} one
instead relies on \eqref{eq:HKU} 
at $y=x$, and on Assumption \ref{Ass:Weight}, noting that
all steps of the 
proof involve only the universal 
$d_f$, $d_w$, $p_0$, $\ce$, $\cv$ and $\cHK$. Further, 
following the proof of \cite[Theorem 5.4]{CG98} it now 
suffices to take only $r \le R_N$,
hence $t \le \eta \, T_N$ for some fixed $\eta<\infty$.
\end{proof}

\medskip
\begin{proof}[Proof of Proposition \ref{Prop:MTE}]
Recall that $\mu^{(N)}(S) = \pi^{(N)}(S) \mu^{(N)} (V(G^{(N)}))$.
By \eqref{eq:Spec} we further have that 
$\lambda^{(N)}(S) \ge \lambda_1^{(N)} (S)$ for any
choice of $S$ and $N$, hence Lemma \ref{Lem:FKN} results with 
\begin{equation}\label{eq:spec-prof-bd}
\tilde{\Lambda}^{(N)}(r) \ge 
\frac{\cFK}{2} \Big[ r \, \mu^{(N)}(V(G^{(N)})) \Big]^{-d_w/d_f} \,.
\end{equation}
By the assumed $d_f$-set condition, 
$\mu^{(N)}(V(G^{(N)})) \le \cv \, \Diam\{G^{(N)}\}^{d_f}$.
Thus, combining \eqref{eq:Umix-bd} and \eqref{eq:spec-prof-bd}
yields the bound 
\[
T_{\rm{mix}}^U (\epsilon ; G^{(N)} ) \le \frac{8}{\cFK} \frac{d_f}{d_w} 
\Big(\frac{4 \, \cv}{\epsilon}\Big)^{d_w/d_f} \, T_N \,,
\]
as claimed. \end{proof}

We conclude with a very useful covering property. 
\begin{Proposition}      \label{Prop:Covering}
Assumption \ref{Ass:Weight} 
implies that for any $\eta \in (0,1]$, there exist $L = L(\eta,d_f,
\tilde{c} \, \cv) < \infty$ such that each $G^{(N)}$ can be covered by $L$ balls $\{ B^{(N)} (x_i ,\eta \, R_N) \}_{i=1}^{L}$ of $V(G^{(N)})$.  
\end{Proposition}   
\begin{proof} Covering $V(G^{(N)})$ by a single ball of radius 
$R_N$, thanks to \eqref{eq:Card} and 
the assumed $d_f$-set condition  
$\sharp V(G^{(N)}) \le \tilde{c} \, \cv (R_N)^{d_f}$. Further, $G^{(N)}$ can be covered
by $L$ balls $B^{(N)} (x_i ,\eta R_N)$ such that $\{B^{(N)}(x_i,\eta R_N/2)\}$ 
are disjoint (e.g. \cite[Lemma 6.2(a)]{Barlow17}).  
Consequently, $L (\tilde{c} \, \cv)^{-1} (\eta R_N/2)^{d_f} \le \sharp V(G^{(N)})$
and we conclude that $L \le (\tilde{c} \, \cv)^2 (2/\eta)^{d_f}$ for all $N$, as claimed.
\end{proof}

\subsection{Strongly recurrent case: $d_f< d_w$ }  \label{Subsec:Rec}
A consequence of Assumptions \ref{Ass:Weight}, \ref{Ass:UPHI} 
for $d_f < d_w$ is the following relation between 
the resistance metric and the graph distance.
\begin{Proposition}  \label{Prop:ERN} 
Suppose Assumptions \ref{Ass:Weight}, \ref{Ass:UPHI} and $d_f < d_w$. 
Then, for some $\creff$ finite, all $N \ge 1$ and any $x,y \in V(G^{(N)})$,
     \begin{align} \label{eq:ERN} 
              \creff^{-1} \, d^{(N)} (x,y)^{d_w - d_f} \le 
              R_{\rm{eff}}^{(N)} (x,y)  \le  \creff \, d^{(N)} (x,y)^{d_w - d_f} \,.
     \end{align}      
\end{Proposition}  
\begin{proof} (Sketch:) For a single infinite weighted graph this 
is a well known consequence of (HKE($d_w$)), see for example 
\cite[Theorem 1.3]{BCK05}. In our setting, the upper bound on 
$R_{\rm{eff}}^{(N)}$ is derived from Proposition \ref{Prop:HKEN}
by going via (PI($d_w$)), as done in the proof of 
\cite[Lemma 2.3(ii), Proposition 4.2(1)]{BCK05}. The 
corresponding lower bound in \eqref{eq:ERN} is proved
as in \cite[Proposition 4.2(2)]{BCK05}, by showing instead 
the property (SRL$(d_w)$) (see remark at \cite[bottom of Pg. 1650]{BCK05}).
As in Proposition \ref{Prop:HKEN}, all steps use only 
constants from Assumptions \ref{Ass:Weight}--\ref{Ass:UPHI}
and require our \abbr{sub-HKE} only at $t \le \eta_0 T_N$.
Hence, we end with finite $\creff$ which is independent of $N$.  
\end{proof}

The following corollary of 
Proposition \ref{Prop:ERN} is immediate.
\begin{Corollary}  \label{Cor:Scale}
Suppose Assumptions \ref{Ass:Weight} and \ref{Ass:UPHI} hold for some 
$d_f < d_w$ and let
     \begin{align}   \label{eq:Not05}
             r( G^{(N)} ) := \max_{x,y \in V(G^{(N)}) }  \{ R_{\rm{eff} }^{(N)} (x,y) \},  \qquad  
             S_N := \mu^{(N)} (V(G^{(N)})) r (G^{(N)}) . 
     \end{align}
Then, for some finite $c_\star$  
      \begin{align*}    
 c_\star^{-1} T_N  \le S_N  \le c_\star T_N,   \qquad   \forall N \ge 1 \, . 
     \end{align*}
\end{Corollary} 
\begin{proof} By our $d_f$-set condition 
$\mu^{(N)} (G^{(N)}) \asymp (R_N)^{d_f}$, 
whereas $r(G^{(N)}) \asymp  (R_N)^{d_w - d_f}$, 
thanks to Proposition \ref{Prop:ERN}. With 
$T_N := (R_N)^{d_w}$ we are thus done.
\end{proof}

\subsection{Transient case: $d_f > d_w$ }  \label{Subsec:Trans}
When $d_f>d_w$, Proposition \ref{Prop:MTE} and  \eqref{eq:HKU} 
yield the following decay rate of the Green functions.   
 \begin{Proposition}   \label{Prop:Green}
    Suppose  Assumptions \ref{Ass:Weight}, \ref{Ass:UPHI} and 
   $d_f > d_w$. Then, for some $\cgr(\cdot)$ finite, 
   any $\epsilon>0$ and finite $N$,
             \begin{align}   \label{eq:Green}
                    \tilde{g}^{(N)} (x,y) : =  \sum_{t=0}^{T_{\rm{mix}}^U (\epsilon ; G^{(N)} ) }  \tilde{p}_t^{(N)} (x,y) \le \cgr(\epsilon) 
                    \, d^{(N)} (x,y)^{d_w -d_f},  
                           \qquad  \forall y \ne x \in V(G^{(N)}) \, . 
            \end{align}
\end{Proposition} 
\begin{proof} Clearly 
$\tilde{p}_t^{(N)}(x,y)=\sum_{s} q_t(s) p_s^{(N)}(x,y)$
with $q_t(s)$ the probability that a 
Binomial$(t,1/2)$ equals $s$. Consequently,   
$\tilde{g}^{(N)} (x,y) \le 2 g^{(N)} (x,y)$ 
(since $\sum_{t} q_t(s) =2$). We further 
replace $T_{\rm{mix}}^U (\epsilon ; G^{(N)})$ 
in \eqref{eq:Green} by $\eta T_N$, for 
$\eta:=c(\epsilon)$ of Proposition \ref{Prop:MTE}. 
Hence, from \eqref{eq:HKU} for 
some $\cHK=\cHK(\eta)$, all $N$ and $x \ne y$, 
\[ 
\tilde{g}^{(N)}(x,y) \le 2 \cHK \sum_{t=1}^{\infty} 
t^{-d_f/d_w} 
\exp \Big[  -  \cHK^{-1} \Big(  \frac{ d^{(N)} (x,y)^{d_w}}{t}  \Big)^{1/(d_w -1) }  
\Big] \,.
\]
Since $d_f/d_w>1$, the series on the \abbr{rhs} converges 
(even when $d^{(N)}(x,y)=0$), and it is easy to further 
bound it by $\cgr' \, d^{(N)} (x,y)^{d_w -d_f}$ for some 
$\cgr'=\cgr'(\cHK)$ finite, as we claim in \eqref{eq:Green}.
\end{proof}

\section{Cover time: Proof of Proposition \ref{Thm:UCT10}}   \label{Sec:CT}

We recall $S_N$, $r(G^{(N)})$ of \eqref{eq:Not05} and 
use the following notations for $x,y \in V(G^{(N)})$, $r \in [0,1]$, 
\begin{equation*}
\TERN  (x,y) := \frac{ \ER^{(N)} (x,y) }{ r(G^{(N)}) } \in [0,1],   \qquad   B_R^{(N)} (x,r) : = \{ y  \in V(G^{(N)}) \mid \TERN (x,y ) \le r \}\,.  
\end{equation*}
We show in Lemma \ref{Lem:Loc10} that for some $\epsilon'>0$, with 
positive probability, during its first $S_N$ steps, 
a random walk on $G^{(N)}$ makes at least $\epsilon' \, r(G^{(N)})$ visits 
to the starting point. Combining this with the modulus 
of continuity of the relevant local times (of Lemma \ref{Lem:UTELT20}),
we show in Proposition \ref{Prop:Loc30} and Corollary \ref{Cor:Loc50} 
that for some $\kappa>0$, with positive probability, by time 
$4 S_N$ a (small) ball $B^{(N)}_R (x,\kappa)$ is covered by the
random walk trajectory. In view of Propositions \ref{Prop:Covering}
and \ref{Prop:ERN}, if in addition $d_f<d_w$, then for some 
$L=L(\kappa,\creff)$ finite and all $N$, the set $V(G^{(N)})$ is covered 
by some $\{ B^{(N)}_R (z_i, \kappa) \}_{i=1}^L$.
Proposition \ref{Thm:UCT10} then follows by using this fact,
the Markov property and having 
$S_N \asymp T_N$ (see Corollary \ref{Cor:Scale}).

\medskip
We now implement the details of the preceding proof strategy.
\begin{Lemma}  \label{Lem:Loc10}
Under Assumptions \ref{Ass:Weight} and \ref{Ass:UPHI}, 
there exists $\epsilon>0$ such that
        \begin{align}\label{eq:Loc10-bd}
        \max_{N \ge 1} \max_{x \in V(G^{(N)})}
                 P_x \left(  \widehat{L}^{(N)}_{S_N} (x)  \le 2 \epsilon  \right)  
                            \le \frac{1}{8} \,, 
                            \quad \qquad 
 \widehat{L}_t^{(N)}(x) := \frac{1}{r(G^{(N)}) \mu_x^{(N)}}  \sum_{s=0}^{t-1}  1_{x} ( X_s^{(N)} ) \,.           
        \end{align}   
 \end{Lemma}
\begin{proof} Recall that the successive times  
in which the walk $X^{(N)}_t$ re-visits $x=X^{(N)}_0$, 
form a partial sum, whose i.i.d. $\mathbb{N}$-valued increments 
$\{\eta_x^{(N)}(i)\}_{i \ge 1}$ have mean
\[
E_x \big[\eta_x^{(N)} \big] = \frac{1}{\pi^{(N)}(x)} =  
\frac{ \mu^{(N)} (G^{(N)} ) }{ \mu_x^{(N)}} \,.
\] 
Setting $m_x^{(N)}:=[2 \epsilon \, \mu_x^{(N)} \, r (G^{(N)})]$ we 
thus have by Markov's inequality that
\begin{align*}
P_x \left( \widehat{L}_{S_N}^{(N)} (x) \le  2 \epsilon  \right)  
& = P_x  \Big(  \sum_{i=1}^{m_x^{(N)}} \eta_x^{(N)} (i) \ge  S_N  \Big)    
    \le  \frac{m_x^{(N)}}{S_N}  E_x \big[ \eta_x^{(N)} \big]
                    \le  2 \epsilon \,, 
\end{align*} 
yielding  \eqref{eq:Loc10-bd} when $\epsilon \le 2^{-4}$. 
\end{proof}

\medskip
With our graphs having uniform volume growth, \cite[Theorem 1.4]{C15}
applies here, giving the following modulus of continuity result.
\begin{Lemma}   \label{Lem:UTELT20}   
Suppose Assumptions \ref{Ass:Weight} and \ref{Ass:UPHI}.
Then, for $\varphi(\kappa) := \sqrt{\kappa (1 + |\log \kappa|)}$ we have that 
   \begin{align*}
\Delta (\lambda) := 
\sup_{\kappa \in (0,1], N \ge 1} \sup_{z \in V(G^{(N)})}
 P_z \Big( 
 \max_{t \le S_N}  \max_{\substack{x,y \in V(G^{(N)}) \\ \TERN  (x,y) \le \kappa }}
 |  \widehat{L}^{(N)}_{t} (x)  - \widehat{L}_{t}^{(N)}  (y) | 
                        \ge \lambda \varphi(\kappa)  \Big) 
                        \to 0 \,, \; \hbox{ as } \; \lambda \to \infty \,.
   \end{align*}
\end{Lemma}

\noindent
Combining Lemmas \ref{Lem:Loc10} and \ref{Lem:UTELT20} yields
the following uniform lower bound on the minimum 
over $y \in B_R^{(N)}(x,\kappa)$, of the
normalized local time at $y$ during the first $4 S_N$ moves of 
the random walker. 
\begin{Proposition}   \label{Prop:Loc30}
Under Assumptions \ref{Ass:Weight} and \ref{Ass:UPHI},
for some positive $\epsilon, \kappa$ 
        \begin{align}   \label{eq:Loc31}
              \inf_{N \ge 1}  \inf_{x,z \in V(G^{(N)})} 
               P_z \Big(  \min_{y \in B_R^{(N)} (x,\kappa)}
               \{ \widehat{L}_{4 S_N }^{(N)} (y) \} \ge \epsilon  \Big) \ge \frac{1}{2}\,.  
        \end{align} 
\end{Proposition}

\begin{proof}  
{\sf Step 1.}  Taking $\epsilon>0$ as in Lemma \ref{Lem:Loc10}, we
 first show that for some $\kappa>0$,      
        \begin{align*}
              \inf_{N \ge 1}  \inf_{x \in V(G^{(N)})} 
               P_x \Big(  \min_{y \in B_R^{(N)}(x,\kappa)} 
                        \left\{  \widehat{L}_{S_N}^{(N)} (y) \right\}   
                        \ge \epsilon \,\, \Big)  
                            \ge \frac{3}{4} \,.
        \end{align*} 
To this end considering
Lemma \ref{Lem:UTELT20} for $\lambda<\infty$ such that
$\Delta (\lambda) < 2^{-3}$
and $\kappa >0$ such that $\lambda \varphi(\kappa) \le \epsilon$,
 we obtain that, for all $N$ and any $z \in V(G^{(N)})$,
      \begin{align*}
               P_z \Big( \max_{x,
              y \in B_R^{(N)}(x,\kappa)}   
             \Big\{ |\widehat{L}_{S_N }^{(N)} (x) - \widehat{L}_{S_N
             }^{(N)} (y) | \Big\} \ge \epsilon \Big)  \le \frac{1}{8}\,.
      \end{align*}
Consequently, by Lemma \ref{Lem:Loc10},
    \begin{align*} 
           \frac{7}{8} \le  
           P_x \Big(  \widehat{L}_{ S_N }^{(N)} (x) 
          \ge 2 \epsilon \Big) & \le \frac{1}{8}  
                    +  P_x  \Big(  \widehat{L}_{ S_N }^{(N)} (x)  \ge 2 \epsilon,     
                             \max_{
              y \in B_R^{(N)}(x,\kappa)}   
            \Big\{ |\widehat{L}_{ S_N }^{(N)} (x) - \widehat{L}_{ S_N }^{(N)} (y) |  \Big\} \le \epsilon \Big)    \\  
           & \le \frac{1}{8}
                    +  P_x  \Big(  \min_{y \in B_R^{(N)}(x,\kappa)} 
                       \Big\{ \widehat{L}_{S_N}^{(N)} (y)   
                        \Big\} \ge \epsilon \Big)\,,     
    \end{align*} 
thereby completing Step 1.  

\noindent{\sf Step 2.} Turning to prove \eqref{eq:Loc31} when $z \ne x$, 
let $\tau_x^{(N)} := \inf \{ t \ge 0 \mid X_t^{(N)} = x \}$  
denote the first hitting time of $x \in V(G^{(N)})$ by the random walk.
Recall the commute time identity (see \cite[Proposition 10.6]{LPW08}), 
that for any $N$ and $x \ne z$ in $V(G^{(N)})$,
\begin{align}\label{eq:TVest20}
E_x \big[  \tau^{(N)}_z \big] + 
E_z \big[  \tau^{(N)}_x \big] =
\ERN (z,x)  \mu^{(N)}(G^{(N)}) \,.
\end{align}
Hence,  
\begin{align}\label{eq:Loc34}
P_z \left( \tau_x^{(N)} \ge 3 S_N \right)  
& \le   \frac{1}{3 \, S_N} E_z \big[  \tau^{(N)}_x \big] 
\le \frac{1}{3} 
          \end{align}             
so by the strong Markov property
at $\tau_x^{(N)}$, we see that for any $z \in V(G^{(N)})$, 
      \begin{align*}
           P_z  \Big( 
          \min_{y \in B_R^{(N)}(x,\kappa)} 
          \{ \widehat{L}^{(N)}_{4 S_N} (y) \} \ge \epsilon
                         \Big)       
          &  \ge   \sum_{t=0}^{3 S_N} 
          P_z  \Big( \min_{y \in B_R^{(N)}(x,\kappa)}
           \{ \widehat{L}^{(N)}_{4 S_N} (y) - 
           \widehat{L}^{(N)}_{t} (y) \}  \ge \epsilon  \,, \; 
         \tau^{(N)}_x = t \Big) \\    
          &   =   
          \sum_{t=0}^{3 S_N} P_z  \big(   
          \tau^{(N)}_x =t) P_x \Big( 
          \min_{y \in B_R^{(N)}(x,\kappa)}
          \Big\{ \widehat{L}^{(N)}_{4 S_N - t} (y) 
          \Big\} \ge \epsilon 
                          \Big)      \\
          & \ge
          P_z \left(  \tau^{(N)}_x \le 3 S_N \right)  
P_x  \Big(  
                    \min_{y \in B_R^{(N)}(x,\kappa)} 
                       \Big\{ \widehat{L}^{(N)}_{S_N} (y)   
                       \Big\} \ge \epsilon
                     \Big)        
             \ge \frac{1}{2} \,,    
      \end{align*}
by combining Step 1 and \eqref{eq:Loc34}.  
\end{proof}

Denoting the range of the random walk 
by ${\mbox{Range}}^{(N)}_t := \{ X_0^{(N)},  X_1^{(N)}, \ldots,  
X_{t-1}^{(N)} \}$, 
we have the following consequence of Proposition \ref{Prop:Loc30}.
\begin{Corollary}  \label{Cor:Loc50}
If Assumptions \ref{Ass:Weight} and \ref{Ass:UPHI} hold, 
then for some $\kappa>0$ and any $t$,     
\begin{align*}
\sup_{N \ge 1} \sup_{x,z \in V(G^{(N)})} 
P_z \left(  {\mbox{Range}}_{t}^{(N)}   \not\supseteq B_R^{(N)} (x, \kappa)   \right) \le 2^{1 - t/(4 S_N)} \,.
\end{align*}
\end{Corollary}
\begin{proof} Taking $\kappa>0$ as in 
Proposition \ref{Prop:Loc30}, we have that 
for all $N$ and $x, z \in V(G^{(N)})$,
             \begin{align*}
P_z  \left(   {\mbox{Range}}_{4 S_N}^{(N)} \supseteq B_R^{(N)} (x,\kappa) \right)  \ge \frac{1}{2} \,. 
\end{align*}  
Applying
the Markov property at times $\{4 i S_N\}$ for $i=1,\ldots,k-1$, it follows that 
\begin{align*}
P_z  \left(   {\mbox{Range}}_{4 k S_N }^{(N)}   \not\supseteq B_R^{(N)} (x, \kappa)   \right)  \le 2^{-k} 
           \end{align*}  
and we are done, since $t \mapsto {\mbox{Range}}^{(N)}_t$ is non-decreasing. 
\end{proof}

\begin{proof}[Proof of Proposition \ref{Thm:UCT10}]    
From Proposition \ref{Prop:ERN}, if 
$\creff^2 \, \eta^{d_w-d_f} \le \kappa$, then for any
$N$ and $x \in V(G^{(N)})$,
\[
B^{(N)}(x,\eta \, R_N) \subseteq B^{(N)}_R(x,\kappa) \,.
\]
Setting such $\eta=\eta(\creff,\kappa)>0$ we deduce from 
Proposition \ref{Prop:Covering} that for any $\kappa>0$ there
exist $L=L(\kappa)$ finite
and $x_1, \ldots , x_L \in V(G^{(N)})$, such that for all $N$,
\begin{align*}
V(G^{(N)} )  = \bigcup_{i=1}^L  B_R^{(N)} (x_i,\kappa) \,.  
\end{align*}
We embed the walk $X^{(N)}_s$ within the sample path $s \mapsto \tilde{X}^{(N)}_s$ 
of its lazy counterpart, such that the number of steps $M_t$ made 
by the lazy walk during the first $t$ steps of $\{X^{(N)}_s\}$ 
is the sum of $t$ i.i.d. Geometric($1/2$) variables, 
which are further independent of $\{X^{(N)}_s\}$. Since the
range of the lazy random walk at time $M_t$ is then
$\mbox{Range}^{(N)}_{t}$, we have for any $t$, $N$ and $z \in V(G^{(N)})$ 
\begin{align*} 
P_z  \left( \tau_{\rm{cov}} (G^{(N)}) > 3 t  \right)     
&  \le P (M_t > 3 t) + 
 \sum_{i=1}^L  P_z  \left(  {\mbox{Range}}_{t}^{(N)} \not\supseteq B_R^{(N)} (x_i, \kappa)   \right)  \,.
\end{align*}
By Cramer-Chernoff bound, the first term on the \abbr{rhs} is at
most $\theta^t$ for some $\theta<1$. With $L=L(\kappa)$ independent of $N$, $z$, and
$S_N \le c_\star T_N$ (see Corollary 
\ref{Cor:Scale}), we thus reach 
\eqref{eq:tcov-spread} upon choosing $\kappa>0$ 
as in Corollary \ref{Cor:Loc50} and 
$c_0 \ge 2L(\kappa)+1$ such that $e^{-3/c_0} \ge \max(\theta,2^{-1/(4c_\star)})$. 
\end{proof}

\section{Lamplighter mixing: Theorem 
\ref{Thm:Main} and Proposition \ref{Prop:LBTV10}}   \label{Sec:Rec}
  
\begin{proof}[Proof of Proposition \ref{Prop:LBTV10}]
\abbr{wlog} we may and do assume that $\bm{x_0} = (\bm{0},x_0)$ for 
some $x_0 \in V(G^{(N)})$. Let 
       \begin{align*}
A^{\ast}_N := \Big\{ (f,x) \in  V( \mathbb{Z}_2 \wr G^{(N)}) \; \big| \;
                           \exists y \in V(G^{(N)}) 
                         \text{ such that } f(b) \equiv 0, \quad     
 \forall b \in B^{(N)} (y, r_N) 
                        \Big\} \,,
       \end{align*}
where taking $r_N := \lceil (2d_f \tilde{c} \, \cv  \log_2 R_N)^{1/d_f} \rceil$ 
we have thanks to \eqref{eq:Card} and the $d_f$-set condition, that 
\[
\sharp B^{(N)}(y,r_N) \ge \tilde{c}^{-1} V^{(N)}(y,r_N) \ge 
(\tilde{c} \, \cv)^{-1} (r_N)^{d_f} \ge 2 d_f \log_2 R_N \,.
\]
By the same reasoning
$\sharp  V(G^{(N)}) \le \tilde{c} \, \cv (R_N)^{d_f}$, 
so for the invariant distribution
$\pi^{\ast} ( \cdot; G^{(N)})$ of 
the lamplighter chain $Y^{(N)}$ on $\mathbb{Z}_2 \wr G^{(N)}$
\begin{equation}\label{eq:LBTV14}
\pi^{\ast} (A^{\ast}_N ; G^{(N)})   \le   \sum_{y \in V(G^{(N)})} 
2^{-\sharp B^{(N)}(y,r_N)} \le
\tilde{c} \, \cv (R_N)^{-d_f} \,. 
\end{equation}
Part of our $d_f$-set condition is having $R_N \to \infty$, so 
there exists $N_1$ finite such that $R_0 \le r_N \le \frac{1}{4} R_N$ 
for $R_0$ of Corollary \ref{Cor:HKEN} and any $N \ge N_1$.  
Since $\max_y \{ d^{(N)}(x_0,y)\} \ge \frac{1}{2} R_N$ for any 
$x_0 \in V(G^{(N)})$, whenever $N \ge N_1$ the event
\[
\tilde{\Gamma}^{(N)}_t :=  \Big\{  \max_{0 \le s \le t} d(\tilde{X}_0^{(N)}, \tilde{X}_s^{(N)}) \le \frac{1}{4} R_N \Big\}
\]
implies that $\{ Y_t^{(N)} \in A^{\ast}_N \}$. Consequently, for any 
such $N$ we have by \eqref{eq:LBTV14}
that 
\begin{align} \label{eq:LBTV12}
                \max_{\bm{x} \in V(\mathbb{Z}_2 \wr G^{(N)}) }  \| P_t^{\ast}  ( \bm{x}, \cdot ; G^{(N)} ) - \pi^{\ast} (\cdot ; G^{(N)}) \|_{\TV}  
                  & \ge      P_{ \bm{x}_0}^{\ast}  ( Y_t^{(N)} \in A^{\ast}_N ; G^{(N)}) - \pi^{\ast} (A^{\ast}_N; G^{(N)})    \notag    \\
               &    \ge       P_{x_0} (\tilde{\Gamma}^{(N)}_t;G^{(N)}) - 
\tilde{c} \, \cv (R_N)^{-d_f} \, . 
       \end{align}
Let $c_1 := 4^{d_w} c_2$ for $c_2<\infty$ of Corollary \ref{Cor:HKEN}.
Then, by Corollary \ref{Cor:HKEN} at  
$r=\frac{1}{4} R_N$, we have for all $N \ge N_1$ 
\begin{align*} 
             & P_{x_0} \big(  
             \tilde{\Gamma}^{(N)}_{t}; G^{(N)}
              \big)  
                  \ge  P_{x} \Big(   \max_{0 \le s \le t} 
                  d^{(N)} ( X_0^{(N)}, X^{(N)}_s)  \le \frac{1}{4}  R_N \Big)      
                \ge c_1^{-1} e^{- c_1 t/T_N} \,, 
       \end{align*}
which together with \eqref{eq:LBTV12} completes the proof.
\end{proof}

As shown next, at $t \gg S_N$ the lazy walk is near 
equilibrium (in total variation), and the total 
variation distance of $P^{\ast}_t (\bm{x}, \cdot; G^{(N)})$ 
from its equilibrium law
is then 
controlled by the tail probabilities of $\SCT (G^{(N)})$.
\begin{Proposition}   \label{Prop:LampTV} 
For any $t$, weighted graphs $(G^{(N)},\mu^{(N)})$
and $\bm{x} \in V(\mathbb{Z}_2 \wr G^{(N)})$,
\begin{align}
 \| P_t^{\ast}  ( \bm{x} , \cdot ; G^{(N)} ) - \pi^{\ast} (\cdot ;G^{(N)}) \|_{\TV}      
 & \le P_x (\SCT (G^{(N)}) > t) 
 + \| \tilde{P}_t (x, \cdot ; G^{(N)}) - \pi ( \cdot ; G^{(N)}) \|_{\TV}  
  \notag \\
& \le P_x (\SCT (G^{(N)}) > t) + \frac{\sqrt{S_N}}{2 \sqrt{t}}  \,.
\label{eq:tv-walk}
\end{align}  
\end{Proposition}
\begin{proof} Using the uniform 
(invariant) distribution of lamp configurations at
$t \ge \SCT (G^{(N)})$, yields  
\begin{align*}   
\| P_t^{\ast}  ( \bm{x} , \cdot ; G^{(N)} ) - \pi^{\ast} (\cdot ;G^{(N)}) \|_{\TV}  
              & \le \quad
 \sum_{ \bm{y} \in V(\mathbb{Z}_2 \wr G^{(N)}) }  P_{\bm{x}}^{\ast} (Y_t^{(N)} = \bm{y}, \SCT (G^{(N)}) > t)
\notag \\ 
& \quad + \sum_{ \bm{y} \in V(\mathbb{Z}_2 \wr G^{(N)}) }  [P_{\bm{x}}^{\ast} (Y_t^{(N)} = \bm{y}, t \ge \SCT (G^{(N)})) -  \pi^{\ast} (\bm{y} ; G^{(N)} ) ]_+  \notag \\
& \le P_{x} (\SCT (G^{(N)}) > t) +   \sum_{y \in V(G^{(N)})} [ P_x ( \tilde{X}_t^{(N)} = y) -  \pi^{(N)} (y) ]_+  \,.
\end{align*} 
Applying the definition of total variation 
distance for $\tilde{X} = \{ \tilde{X}_t \}_{t \ge 0}$ yields 
the first inequality in \eqref{eq:tv-walk}. Next, let
$\tilde{\tau}_x^{(N)} := \min \{ t \ge 0 \mid \tilde{X}_t^{(N)} = x \}$. 
By the embedding of $X^{(N)}$ within $\tilde{X}^{(N)}$
(as in the proof of Proposition \ref{Thm:UCT10}), 
and the commute time identity (see \eqref{eq:TVest20}), 
we have that for all $N$ and $x, z \in V(G^{(N)})$   
\begin{align}\label{eq:TVest10b}
E_z \big[ \tilde{\tau}^{(N)}_x \big] = 2 E_z \big[  \tau^{(N)}_x \big] \le 2 S_N \,.
\end{align}
While proving \cite[Lemma 4.1]{NP08}, it shown that 
for all $N$, $t$ and $x \in V(G^{(N)})$,   
\begin{align}\label{eq:TVest10}
\Big(\| \tilde{P}_t (x, \cdot ; G^{(N)}) - \pi ( \cdot ; G^{(N)}) \|_{\TV}\Big)^2   
\le  \frac{1}{8 t} \max_{z \in V(G^{(N)})}  
                         \big\{ E_z [  \tilde{\tau}^{(N)}_x] \big\} 
\end{align}
and we get the second inequality in \eqref{eq:tv-walk} by 
combining \eqref{eq:TVest10b} and \eqref{eq:TVest10}. 
\end{proof}   

\subsection{The strongly recurrent case: $d_f < d_w$}  \label{Sec:Recsub}

For $d_f<d_w$ we get Theorem \ref{Thm:Main}(a) by 
combining the lower bounds of Proposition \ref{Prop:LBTV10} 
with the upper bounds of Propositions \ref{Thm:UCT10} and 
\ref{Prop:LampTV}. 
\begin{proof}[Proof of Theorem \ref{Thm:Main}(a).]
Since $R_N \to \infty$, we deduce from
Proposition \ref{Prop:LBTV10} that for any $\epsilon \in (0,1)$,
\begin{equation}\label{eq:lbd-tmix}
\liminf_{N \to \infty} \Big\{ 
\frac{\MT (\epsilon ; \mathbb{Z}_2 \wr G^{(N)})}{T_N} \Big\} \ge 
- c_1^{-1} \log (c_1 \epsilon) \,.
\end{equation}
In contrast, with $S_N \le c_\star T_N$ and 
$\gamma=\gamma(\epsilon)$ denoting the unique solution of 
\begin{equation*} 
\epsilon = c_0 e^{-\gamma/c_0} + \frac{\sqrt{c_\star}}{2\sqrt{\gamma}} \,,
\end{equation*}
we get 
from Propositions \ref{Thm:UCT10} and \ref{Prop:LampTV} that 
\begin{equation}\label{eq:ubd-tmix}
\limsup_{N \to \infty} \Big\{ 
\frac{\MT (\epsilon ; \mathbb{Z}_2 \wr  G^{(N)})}{T_N} \Big\} \le \gamma(\epsilon) \,.
\end{equation}
The \abbr{rhs} of \eqref{eq:lbd-tmix} blows up as $\epsilon \to 0$, 
while the \abbr{rhs} of \eqref{eq:ubd-tmix} is uniformly bounded 
above for $\epsilon \in [\frac{1}{2},1]$. Hence, there can be no 
cutoff for these lamplighter chains. 
\end{proof}
\begin{Remark}\label{rem:mix}
In view of Proposition \ref{Thm:UCT10}, here
$\MT(\epsilon; \mathbb{Z}_2 \wr G^{(N)})/\CT (G^{(N)}) \gg 1$ for small 
$\epsilon$. From Section \ref{Sec:CT} 
we also learn that, when $d_f<d_w$, the lamplighter chains 
have no mixing cutoff 
mainly 
because the laws of
$\tau_{\rm{cov}} (G^{(N)})/\CT(G^{(N)})$
do not concentrate as $N \to \infty$ (unlike
the transient case of $d_f>d_w$).
\end{Remark}

\subsection{The transient case: $d_f > d_w$}      \label{Sec:Tran}

As mentioned before, in case $d_f>d_w$, we establish the cutoff 
for total-variation mixing time of the lamplighter chains by
verifying that our weighted graphs 
$\{ (G^{(N)},\mu^{(N)})  \}_{N \ge 1}$ satisfy the 
sufficient conditions from \cite[Theorem 1.5]{MP12}.  To this end,  
recall the uniform mixing times $\MT^U (G^{(N)})$ and 
Green functions $\tilde{g}^{(N)}(\cdot,\cdot)$ that correspond
to $\epsilon=\frac{1}{4}$ in \eqref{eq:Umix} and
\eqref{eq:Green}, respectively. In \cite{MP12},
uniformly elliptic, finite weighted graphs 
$\{(G^{(N)},\mu^{(N)}) \}_{N \ge 1}$ are called 
\emph{uniformly locally transient} if for all $N$,
\begin{align*}   
g (x,A;G^{(N)}) := \sum_{y \in A} \tilde{g}^{(N)} (x,y)  
\le \rho (d^{(N)} (x,A), \Diam \{A\} ) \,,\qquad  
\forall x \in V (G^{(N)}), A \subseteq V(G^{(N)}) \,,
\end{align*} 
where $\rho : \R_+ \times \R_+ \to \R_+$ is such that 
$\rho(r,s) \downarrow 0$ as $r \to \infty$, for each fixed $s$.
Further setting 
\begin{align*}
\bar{ \Delta } (G)  := \max_{x \in V}  \{ \mu_x  \}, \quad  \underline{\Delta} (G) := \min_{x \in V}  \{ \mu_x \},  
\quad  \Delta (G) := \frac{ \bar{ \Delta } (G) }{  \underline{\Delta} (G) }  \,,
\end{align*}  
the following two assumptions are made in \cite{MP12}. 
\begin{Assumption}[Transience] $~$ \label{Ass:Tran}
The finite weighted graphs $\{ (G^{(N)}, \mu^{(N)}) \}_{N \ge 1}$ are such 
that for any fixed $r<\infty$, as $N \to \infty$,
\begin{enumerate}  \renewcommand{\labelenumi}{(\alph{enumi})} 
\item $\mu^{(N)} (G^{(N)})  \to \infty$.  
\item $\sup_{N} \{ \Delta (G^{(N)}) \} < \infty$.    
\item $\sup_{x} \{ \log V^{(N)} (x,r) \} = o (\log \mu^{(N)}  (G^{(N)}))$. 
\item $ \MT^U (G^{(N)})  (\bar{\Delta} (G^{(N)}))^r = o ( \mu^{(N)}  (G^{(N)}))$. 
\end{enumerate}  
\end{Assumption}
\begin{Assumption}[Uniform Harnack inequalities] $~$  \label{Ass:Har}
  For some $C( \alpha) < \infty$ and all  
  $N,r \ge 1$, $\alpha >1$,
  $x \in V(G^{(N)})$, if $h(\cdot)$ is a positive $\mu^{(N)}$-harmonic 
  on $B^{(N)} (x,\alpha r)$, then
          \begin{align*}
                  \max_{y \in B^{(N)} (x,r) } \{ h(y) \} \le  
                  C(\alpha) \min_{ y \in B^{(N)} (x,r) } \{ h (y) \} \,.  
         \end{align*}
\end{Assumption}
We next prove Theorem \ref{Thm:Main}(b) by relying on the following 
restatement of \cite[Theorem 1.5]{MP12}.
\begin{Theorem}\label{Thm:MP}
If uniformly locally transient  
$\{ (G^{(N)}, \mu^{(N)} ) \}_{N \ge 1} $ satisfy 
Assumptions \ref{Ass:Tran} and \ref{Ass:Har}, 
then the lamplighter chains $\{Y^{(N)}\}_{N \ge 1}$ 
have cutoff at the threshold $ \frac{1}{2} \CT (G^{(N)})$.
\end{Theorem}  
\begin{Remark} 
The derivation of \cite[Theorem 1.5]{MP12} is limited to 
lazy \abbr{srw} on graphs $G^{(N)}$, namely with $\mu_{xy} \equiv 1$ for
all $xy \in E(G)$. However, up to the obvious modifications we made 
in Assumptions \ref{Ass:Tran} and \ref{Ass:Har}, 
the same argument applies for uniformly elliptic weighted graphs,
as re-stated in Theorem \ref{Thm:MP}.
\end{Remark}

\begin{proof}[Proof of Theorem \ref{Thm:Main}(b).]
Thanks to Proposition \ref{Prop:Green} and \eqref{eq:Card} we
confirm that $(G^{(N)},\mu^{(N)})$ are uniformly locally 
transient for $\rho(r,s)=\cgr \, \tilde c \, \cv r^{d_w-d_f} s^{d_f}$.
Having $\mu^{(N)}(G^{(N)}) \ge \cv^{-1} (R_N)^{d_f} 
\to \infty$ and $G^{(N)}$ of uniformly bounded degrees (see
Remark \ref{Rem:Card}), conditions (a)-(c) of Assumption 
\ref{Ass:Tran} also hold here. Further, with $d_w < d_f$, the bound 
$\MT^U (G^{(N)}) \le c (R_N)^{d_w}$ of Proposition \ref{Prop:MTE}
yields Assumption \ref{Ass:Tran}(d). Considering 
Assumption \ref{Ass:UPHI} for $u(t,\cdot)=h(\cdot)$ results
with the lazy version $\tilde{P}^{(N)}$ satisfying the 
uniform Harnack inequality of Assumption \ref{Ass:Har}
for any $\alpha > \max(2,1/\cPHI)$. By our $p_0$-condition 
this is equivalent to the full Assumption \ref{Ass:Har} 
(see \cite[Proposition 3.5]{Telcs06}), and we 
complete the proof by applying Theorem \ref{Thm:MP}. 
\end{proof}

\medskip  \medskip  \medskip

\ack ~~
This project was supported by the Kyoto University Top Global University (KTGU) Project. 
The research of Amir Dembo was partially supported by NSF Grant Number DMS-1613091.
Takashi Kumagai was partially supported by JSPS KAKENHI (A) 25247007 and 17H01093. 
Chikara Nakamura was partially supported by JSPS KAKENHI Grant Number 15J02838.


\begin{thebibliography}{99}  
\bibitem{Barlow98} M.T. Barlow, Diffusions on fractals.
    Lectures on probability theory and statistics (Saint-Flour, 1995), 1--121, Lecture Notes in Math., 1690, Springer, Berlin, 1998.

\bibitem{Barlow17} M. T. Barlow,  Random walks and heat kernels on graphs. 
    London Mathematical Society Lecture Note Series, 438. Cambridge University Press, Cambridge, 2017.

\bibitem{BB97} M. T. Barlow  and  R. F. Bass,  Random walks on graphical Sierpinski carpets. Random walks and discrete
potential theory (Cortona, 1997), 26--55, Sympos. Math., XXXIX, Cambridge Univ. Press,
Cambridge, 1999.

\bibitem{BB99} M. T. Barlow and R. F. Bass, Brownian motion and harmonic analysis on Sierpinski carpets. Canad. J. Math. 51 (1999), no. 4, 673--744.

\bibitem{BB04} M. T. Barlow and  R. F. Bass,  Stability of parabolic Harnack inequalities. Trans. Amer. Math. Soc. 356 (2003), no. 4, 1501--1533.

\bibitem{BCK05} M. T. Barlow, T. Coulhon and T. Kumagai,   
     Characterization of sub-Gaussian heat kernel estimates on strongly recurrent graphs. Comm. Pure Appl. Math. 58 (2005), no. 12, 1642--1677.
     
\bibitem{BGK12} M. T.Barlow,  A. Grigor'yan and T.  Kumagai,  
      On the equivalence of parabolic Harnack inequalities and heat kernel estimates. J. Math. Soc. Japan 64 (2012), no. 4, 1091--1146.  

\bibitem{BM18} M. T.Barlow and M. Murugan, 
Stability of the elliptic Harnack inequality. 
Ann. Math. 187 (2018), no. 3, 777--823. 

\bibitem{CG98}  T. Coulhon and A. Grigor'yan, A. Random walks on graphs with regular volume growth. Geom. Funct. Anal. 8 (1998), no. 4, 656--701. 

\bibitem{C15}  D. A. Croydon, Moduli of continuity of local times of random walks on graphs in terms of the resistance metric. 
     Trans. London Math. Soc. 2 (2015), no. 1, 57--79.  

\bibitem{DDMP16} A. Dembo, J. Ding, J. Miller and Y. Peres,  Cut-off for lamplighter chains on tori: Dimension interpolation and phase transition.  
     Preprint, available at {\tt arXiv:1312.4522}.

\bibitem{GMT06} S. Goel, R. Montenegro and P.  Tetali, 
        Mixing time bounds via the spectral profile. Electron. J. Probab. 11 (2006), no. 1, 1--26.

\bibitem{GT01}  A. Grigor'yan and A. Telcs, Sub-Gaussian estimates of heat kernels on infinite graphs. Duke Math. J. 109 (2001), no. 3, 451--510.  

\bibitem{GT02} A. Grigor'yan and A. Telcs,  Harnack inequalities and sub-Gaussian estimates for random walks. Math. Ann. 324 (2002), no. 3, 521--556.

\bibitem{HJ97}  O. H\"{a}ggstr\"{o}m and  J. Jonasson, 
        Rates of convergence for lamplighter processes.   
        Stochastic Process. Appl. 67 (1997), no. 2, 227--249. 


\bibitem{HK04} B. M. Hambly and  T. Kumagai,   
Heat kernel estimates for symmetric random walks on a class of fractal graphs and stability under rough isometries. Proc. of Symposia in Pure Math. 72, Part 2, pp. 233--260, 
Amer. Math. Soc. 2004.


\bibitem{Jones96}  D. O. Jones,  
Transition probabilities for the simple random walk on the Sierpi\'{n}ski graph.
     Stochastic Process. Appl. 61 (1996), no. 1, 45--69.

\bibitem{Kigami01} J. Kigami,  Analysis on fractals. Cambridge Tracts in Mathematics, 143. Cambridge University Press, Cambridge, 2001.

\bibitem{Kumagai14}  T. Kumagai, Random walks on disordered media and their scaling limits. Lecture notes from the 40th Probability Summer School held in Saint-Flour, 2010. Lecture Notes in Mathematics, 2101. 
\'{E}cole d'\'{E}t\'{e} de probabilit\'{e}s
de Saint-Flour. Springer, Cham, 2014.

\bibitem{KN16} T. Kumagai and C. Nakamura, Lamplighter random walks on fractals.  
J. Theor. Probab. 31 (2018), no. 1, 68--92. 

\bibitem{LPW08}  D. A. Levin, Y. Peres and E. L. Wilmer,
 Markov chains and mixing times. 
American Mathematical Society, Providence, RI, 2009.

\bibitem{MP12}  J.  Miller and Y.  Peres, 
   Uniformity of the uncovered set of random walk and cutoff for lamplighter chains. 
  Ann. Probab. 40 (2012), no. 2, 535--577.

\bibitem{NP08} A.  Nachmias and Y. Peres,  Critical random graphs: diameter and mixing time. Ann. Probab. 36 (2008), no. 4, 1267--1286.   


\bibitem{MS17}   J. Miller and P. Sousi,  Uniformity of the late points of random walk on $Z^d_n$ for d $\ge$3. Probab. Theory Related Fields 167 (2017), no. 3-4, 1001--1056.

\bibitem{PR04}  Y. Peres and D.  Revelle, Mixing times for random walks on finite lamplighter groups. Electron. J. Probab. 9 (2004), no. 26, 825--845.

\bibitem{SC97} L. Saloff-Coste, Lectures on finite Markov chains. Lectures on probability theory and statistics (Saint-Flour, 1996), 301-413, Lecture Notes in Math., 1665, Springer, Berlin, 1997. 

\bibitem{Telcs01}  A. Telcs,  Local sub-Gaussian estimates on graphs: the strongly recurrent case. Electron. J. Probab. 6 (2001), no. 22, 33 pp.

\bibitem{Telcs06} A.  Telcs,  The art of random walks. Lecture Notes in Mathematics, 1885. Springer-Verlag, Berlin, 2006. 

\end{thebibliography}
\end{document}